\documentclass{article}
\usepackage{latexsym}
\usepackage{amsmath,amsthm}
\usepackage{enumerate}
\usepackage{amssymb}
\usepackage{upgreek}

\usepackage[margin=1.1in,dvips]{geometry}

\newcommand{\vol}{\textnormal{vol}}

\def\f12{\frac 1 2}

\def\a{\alpha}
\def\b{\beta}
\def\ga{\gamma}
\def\Ga{\Gamma}

\def\ep{\epsilon}

\def\si{\sigma}
\def\Si{\Sigma}
\def\om{\omega}
\def\Om{\Omega}

\def\pa{\partial}
\def\les{\lesssim}

\def\f12{\frac 1 2}

\newcommand{\lap}{\mbox{$\Delta \mkern-13mu /$\,}}

\newcommand{\nabb}{\mbox{$\nabla \mkern-13mu /$\,}}
\newtheorem{thm}{Theorem}

\newtheorem{prop}{Proposition}

\newtheorem{lem}{Lemma}
\newtheorem{cor}{Corollary}
\newtheorem{remark}{Remark}
\newtheorem{Def}{Definition}
\begin{document}
\title{Global stability of solutions to nonlinear wave equations}
\author{Shiwu Yang}
\date{}
\maketitle

\begin{abstract}
We consider the problem of global stability of solutions to a class of
semilinear wave equations
 with null condition in Minkowski space. We give sufficient conditions
on the given solution which guarantees stability. Our stability result can be reduced to
a small data global existence result for a class of
semilinear wave equations with linear terms
$B^{\mu\nu}\pa_\mu\Phi(t, x)\pa_\nu\phi$, $L^\mu(t,x)\pa_\mu\phi$
and quadratic terms $h^{\mu\nu}(t, x)\pa_\mu\phi\pa_\nu\phi$ where
the functions $\Phi(t, x)$, $L^\mu(t, x)$, $h^{\mu\nu}(t, x)$ decay
rather weakly and the constants $B^{\mu\nu}$ satisfy the null
condition. We show the small
data global existence result by using the new approach developed by M. Dafermos and I. Rodnianski. In particular, we prove the
global stability result under weaker assumptions than those imposed
by S. Alinhac.
\end{abstract}
\section{Introduction}
In this paper, we study the behavior of solutions to the Cauchy
problem
\begin{equation}
\label{THEWAVEEQ0}
\begin{cases}
 \Box w=\mathcal{N}(\pa w), \\
 w(0,x)=\Phi_0(x)+\ep \phi_0(x), \quad \pa_t w(0,x)= \Phi_1(x)+\ep \phi_1(x)
\end{cases}
\end{equation}
in Minkowski space with initial data $\Phi_i(x)$, $\phi_i(x)\in
C_0^{\infty}(\mathbb{R}^3)$. The nonlinearity $\mathcal{N}(\pa w)$
is assumed to satisfy the null condition, that is,
$\mathcal{N}(0)=D\mathcal{N}(0)=0$ and the quadratic part of
$\mathcal{N}(\pa w)$ is $A^{\a\b}\pa_\a w\pa_\b w$ with constant
coefficients $A^{\a\b}$ such that $A^{\a\b}\xi_\a \xi_\b=0$ whenever
$\xi_0^2=\xi_1^2+\xi_2^2+\xi_3^2$.

\bigskip

In \cite{alinhac-sls}, S. Alinhac studied the stability of large
solutions to the quasilinear wave equations
\begin{equation*}
\begin{cases}
\Box w+g^{\a\b\gamma}\pa_{\gamma}w\cdot \pa_{\a\b}w=0,\\
w(0,x)=\Phi_0(x), \quad \pa_t w(0,x)= \Phi_1(x)
\end{cases}
\end{equation*}
in Minkowski space, where $g^{\a\b\gamma}$ are constants satisfying
the null condition ( see \cite{klnull}). More specifically, starting
with a global solution $\Phi(t, x)\in
C^\infty(\mathbb{R}^{3+1})$, consider the Cauchy problem with perturbed initial data
$(\Phi(0, x)+\ep\phi_0, \pa_t\Phi(0, x)+\ep\phi_1)$. He
showed that if $\Phi$ satisfies the condition
\begin{equation}
\label{alinhaccond}
|g^{ij\gamma}\pa_{\gamma}\Phi\cdot\xi_i\xi_j|\leq \a_0
\sum\limits_{i=1}^{3}|\xi_i|^2,\quad \sum\limits_{|k|\leq
7}|\Gamma^k\pa \Phi|\leq C_0 (1+t)^{-1}(1+|r-t|)^{-\f12}
\end{equation}
for some positive constants $\a_0<1$ and $C_0$, then the solution exists globally and is close to $\Phi$. Here $\Gamma$
denotes the collection of Lorentz vector fields,
 see \cite{klinvar}. The problem of global stability of $\Phi$ can be reduced to the following small data
 Cauchy problem
\begin{equation*}
\begin{cases}
\Box \phi+g^{\a\b\gamma}\pa_{\gamma}\phi\cdot
\pa_{\a\b}\phi+g^{\a\b\ga}\pa_\a \Phi\pa_{\b\ga}\phi+g^{\a\b\ga}
\pa_{\b\ga}\Phi\pa_\a\phi=0,\\
\phi(0,x)=\ep \phi_0(x), \quad \pa_t \phi(0,x)= \ep \phi_1(x)
\end{cases}
\end{equation*}
 with given function $\Phi$ satisfying condition \eqref{alinhaccond}. The approach in \cite{alinhac-sls}
relies on the vector field method. In particular, S. Alinhac used the scaling vector field $S=t\pa_t +r\pa_r$
 with weights
 growing in $t$ as commutators. The use of such weighted vector
 fields requires one to make the rather strong assumption that the given
 solution $\Phi(t, x)$ decays uniformly in time $t$ as in \eqref{alinhaccond}.

\bigskip

In this paper, we use the approach developed in \cite{newapp}, \cite{yang1} to
treat the problem of global stability of solutions to nonlinear wave equations. we use a new
method for proving decay for linear problem, developed by M.
Dafermos and I. Rodnianski in \cite{newapp}. This new method avoids the use of vector fields containing positive weights in $t$,
e.g., $S=t\pa_t+r\pa_r$, $L_i=x_i\pa_t+t\pa_i$. Traditionally, the vector fields from the set $\Ga$, including $S$, $L_i$, are 
used as multipliers or commutators. The new approach only commutes the equation with $\pa_t$, $x_i\pa_j-x_j\pa_i$ and allows us 
to obtain the stability results under conditions on $\Phi$ weaker than those
imposed by inequalities \eqref{alinhaccond}. We now describe the
assumptions and the main results.

\bigskip

We denote $(\pa_t, \pa_{x_1}, \pa_{x_2}, \pa_{x_3})$ by $\pa$ and
$(\pa_{x_1}, \pa_{x_2}, \pa_{x_3})$ by $ \nabla$ under the
coordinates $(t, x_1, x_2, x_3)$. We also use the null coordinates
$u=\frac{t-r}{2}$, $v=\frac{t+r}{2}$ defined by the standard polar
coordinates $(t, r, \om)$ in Minkowski space. The vector fields,
used as commutators , are
\[
 Z=\{ \Om_{ij}, \pa_t=T\},\quad \Om_{ij}=x_i\pa_j-x_j\pa_i,
\]
where Greek indices run from 0 to 3 while the Latin indices run from
1 to 3.

\begin{Def}
\label{weakwave}
 We call $\Phi\in C^{\infty}(\mathbb{R}^{3+1})$ a $(\delta, \a, t_0, R_1, C_0)$-\textbf{weak wave} if
 \begin{enumerate}[(i):]
                   \item  $|\pa\Phi(t, x)|\leq C_0, \quad t\leq
                   t_0$,
                   \item $|\pa\Phi(t, x)|\leq C_0(1+r)^{-\f12}(1+(t-|x|)_+)^{-\f12-4\a},
                   \quad |x|\geq R_1,\quad t\geq t_0$,
                   \item $|\pa_v \Phi(t, x)|\leq C_0 (1+r)^{-1-3\a}, \quad |x|\geq R_1 , \quad t\geq t_0
                   $,
                   \item $|\pa \Phi(t, x)|\leq \delta\a(1+r)^{-1-\a}, \quad |x|\leq R_1 , \quad t\geq t_0$
                 \end{enumerate}
for some positive constants $\delta, \a, t_0, R_1, C_0$, where
$\pa_v=\pa_t+\pa_r$. Here $(t-|x|)_+=\max\{0, t-|x|\}$. Without loss
of generality, we assume $\a\leq \frac{1}{4}$ and $R_1\leq t_0$.
\end{Def}
\begin{remark}
Solution of a free wave equation in Minkowski space $\Box \Phi=0$
with compactly supported initial data decays uniformly in time $t$ and is always a $(\delta, \a, t_0,
R_1, C_0)$-weak wave for some constants $\delta, \a, t_0, R_1, C_0$
. We remark here that a
weak wave does not have to decay uniformly in time $t$ in the
cylinder $\{(t, x)||x|\leq R_1\}$.
\end{remark}

In our argument, we estimate the decay of the solution with respect
to the foliation $\Si_{\tau}$, defined as
\begin{align*}
&S_\tau:=\{u=u_\tau, v\geq v_\tau\},\\
&\Si_\tau:=\{t=\tau, r\leq R\}\cup S_\tau,
\end{align*}
where $u_\tau=\frac{\tau-R}{2}$, $v_\tau=\frac{\tau+R}{2}$. The
radius $R$ is a to-be-fixed constant. The corresponding energy flux
is
$$ E[\phi](\tau):=\int_{r\leq R}|\pa\phi|^2dx +
\int_{S_\tau}\left(|\pa_v\phi|^2+|\nabb\phi|^2\right)r^2dvd\om,
$$
where $\nabb$ denotes the induced covariant derivative on the sphere
of radius $r$. We denote
\[
 E_{0}=\sum\limits_{|k|\leq 4}\int_{\mathbb{R}^{3}}|\pa Z^k\tilde\phi(0,
 x)|^2dx,
\]
where $\tilde\phi(0, x)=\phi_0(x)$, $\pa_t\tilde\phi(0, x)=\phi_1(x)$. Here $k$ stands for multiple indices,
 namely if $k=(k_1, k_2)$, then $Z^k=\Om^{k_1}T^{k_2}$, $\Om=
\Om_{ij}$. And if $k\leq \tilde{k}$, then $k_1\leq \tilde{k}_1$,
$k_2\leq \tilde{k}_2$.

In addition to the assumption that the nonlinearity $\mathcal{N}(\pa
w)$ satisfies the null condition, we assume $\mathcal{N}$ is smooth
and
\[
\mathcal{N}(\pa\Phi+\pa\phi)=\mathcal{N}(\pa\Phi)+A^{\mu\nu}\pa_\mu\Phi\pa_\nu\phi+\mathcal{N}^{\mu}(\pa\Phi)\pa_\mu\phi
+\mathcal{N}^{\mu\nu}(\pa\Phi)\pa_\mu\phi\pa_\nu\phi+O(|\pa\phi|^3),
\]
when $\pa\phi$ is small. The coefficients
$\mathcal{N}^\mu(\pa\Phi)$, $\mathcal{N}^{\mu\nu}(\pa\Phi)$ satisfy
\begin{equation}
\label{HH}
\begin{split}
& |Z^\b\mathcal{N}^{\mu}(\pa\Phi)|\leq
C(\pa\Phi)\sum\limits_{|\b'|\leq |\b|}|Z^{\b'}\pa\Phi|^{2+\a_0},\quad \forall |\b|\leq 4,\\
&|Z^\b\mathcal{N}^{\mu\nu}(\pa\Phi)|\leq
C(\pa\Phi)\sum\limits_{|\b'|\leq
|\b|}|Z^{\b'}\pa\Phi|^{\a_0},\quad\forall |\b|\leq 4
\end{split}
\end{equation}
for some positive constant $\a_0$. The constant $C(\pa\Phi)$ depends
only on $\sum\limits_{|\b|\leq4}\|Z^\b \pa\Phi\|_{C^0}$.

We now state our main results.

\begin{thm}
\label{maintheorem} Suppose the nonlinearity $\mathcal{N}(\pa w)$
satisfies the null condition and condition \eqref{HH}. Let $\Phi\in
C^{\infty}(\mathbb{R}^{3+1})$ be a solution of \eqref{THEWAVEEQ0}
when $\ep=0$. Assume $Z^k\Phi$ is $(\delta, \a, t_0, R_1, C_0)$-weak
wave, $\forall |k|\leq 4$. Suppose the initial data $\phi_0(x),
\phi_1(x)$ are smooth and supported in $\{|x|\leq R_0\}$. Then there
exists $\delta_0>0$, depending only on the
  constants $A^{\a\b}$, and $\ep_0>0$, depending on $E_{0}$, $R_0$, $A^{\a\b}$, $\a_0$, $\a$, $t_0$, $R_1$,
  $C_0$,
  such that for any $\delta<\delta_0$, $\ep<\ep_0$, there exists a unique global smooth solution
  $w$ of equation \eqref{THEWAVEEQ0} with the property that $\exists$ positive constant $R$,
  depending on $t_0$, $\a$, $\a_0$, $R_1$, $C_0$, $R_0$, such that for the foliation $\Si_\tau$ with radius $R$,
  the difference $\phi=w-\Phi$ satisfies
\begin{itemize}
\item[(1)] Energy decay
$$E[\phi](\tau)\leq C E_0\ep^2  (1+\tau)^{-1-\f12\a'},\quad \a'=\min\{\frac{\a_0}{6},
\a\}.
$$
\item[(2)] Pointwise decay:
 \begin{align*}
&\qquad\quad|\phi|\leq C\sqrt{E_0}\ep (1+r)^{-1},\\
&\sum\limits_{|\b|\leq 2}|\pa^\b\phi|\leq C \sqrt{E_0}\ep
(1+r)^{-\f12}(1+|t-r+R|)^{-\f12-\frac{\a'}{4}},\quad
\a'=\min\{\frac{\a_0}{6}, \a\},
\end{align*}
\end{itemize}
where $C$ depends on $R$, $\a_0$, $\a$, $t_0$, $R_1$, $C_0$.
\end{thm}
\begin{remark}
The weak decay of $\pa\Phi$ in the spatial direction (
$(1+|x|)^{-\f12}$) excludes general cubic nonlinearities of
$\mathcal{N}(\pa\Phi)$( cubic nonlinearities satisfying the null
condition are allowed). However if condition $(ii)$ in the
definition of \textbf{weak wave} $\Phi$ is improved to
\begin{equation*}
 \sum\limits_{|\b|\leq 4}|\pa Z^{\b}\Phi|\leq
 C_0(1+r)^{-\f12-\a}(1+(t-|x|)_+)^{-\f12-4\a},
\end{equation*}
then it is sufficient to assume
\[
|Z^\b\mathcal{N}^{\mu}(\pa\Phi)|\leq
C(\pa\Phi)\sum\limits_{|\b'|\leq |\b|}|Z^{\b'}\pa\Phi|^{2},\quad
\forall |\b|\leq 4.
\]
This allows any cubic( or higher) nonlinearity of $\mathcal{N}(\pa
w)$.
\end{remark}

\bigskip

Since $\Phi(t, x)$ solves \eqref{THEWAVEEQ0} for $\ep=0$, the
problem of global stability of $\Phi$ is then reduced to the
following small data Cauchy problem
\begin{equation}
\label{THEWAVEEQ}
\begin{cases}
\Box\phi+N(\Phi, \phi)+L(\pa\phi)=F(\pa\phi),\\
\phi(0, x)= \ep\phi_0(x), \quad \phi_t(0, x)=\ep\phi_1(x),
\end{cases}
\end{equation}
where $N(\Phi, \phi)=B^{\a\b}\pa_\a\Phi\cdot\pa_\b\phi$,
$L(\pa\phi)=L^\mu(t, x)\pa_\mu\phi$. The nonlinearity $F(\pa\phi)$
is of the form
\begin{align*}
&F(\pa\phi)=A^{\mu\nu}\pa_\mu
\phi\pa_\nu\phi+Q(\pa\phi)+\textnormal{cubic and higher order terms
of }
\pa\phi,\\
&Q(\pa\phi)=h^{\mu\nu}(t,x)\pa_\mu\phi\pa_\nu\phi.
\end{align*}
Here $A^{\mu\nu}$, $B^{\mu\nu}$ are constants satisfying the null
condition \cite{klnull} and  $\Phi(t, x)$, $L^\mu(t, x)$,
$h^{\mu\nu}(t,x )$ are given functions. For the stability problem,
we have $B^{\mu\nu}=-2A^{\mu\nu}$ and $\phi=w-\Phi$. However, it is
of independent interest to consider the above small data Cauchy
problem with linear terms $N(\Phi, \phi)$, $L^{\mu}\pa_\mu\phi$ and
quadratic terms $h^{\mu\nu}(t, x)\pa_\mu\phi\pa_\nu\phi$ where the
functions $\Phi(t, x)$, $L^\mu(t, x)$, $h^{\mu\nu}(t, x)$ decay
rather weakly, given as follows:

For positive constants $\delta$, $\a$, $t_0$, $R_1$, $C_0$, we
assume $Z^k\Phi$ is $(\delta, \a, t_0, R_1, C_0)$-weak wave,
$\forall |k|\leq 4$ and
\[
|\pa^2 Z^\b \Phi|\leq C_0,\quad \forall |\b|\leq 2.
\]
Similarly, we assume
 \[
 |\pa Z^\b L^\mu|\leq C_0, \quad \forall |\b|\leq 2.
 \]
For $t\leq t_0$, we assume $Z^\b L^\mu(t, x)$, $Z^\b
h^{\mu\nu}(t,x)$ are bounded, that is,
\[
|Z^\b L^\mu(t, x)|+|Z^\b h^{\mu\nu}(t, x)|\leq C_0,\quad \forall
|\b|\leq 4, \quad \forall t\leq t_0.
\]
For $t\geq t_0$, we assume
\[
|Z^{\b}h^{\mu\nu}(t, x)|\leq C_0 (1+|x|)^{-\frac{3}{2}\a},\quad
\forall |\b|\leq 4, \quad \forall t\geq t_0
\]
and $L^\mu(t, x)$ satisfies \textbf{one} of the following two
conditions
\begin{equation}
\label{Lcond1} |Z^{\b}L^\mu (t, x)|\leq \delta \a
(1+|x|)^{-1-3\a},\quad \forall |\b|\leq 4,\quad \forall t\geq t_0
\end{equation}
or
\begin{equation}
\label{Lcond2} |Z^{\b}L^\mu (t, x)|\leq C_0
(1+|x|)^{-1-3\a}(1+(t-|x|)_{+})^{-\a},\quad \forall |\b|\leq 4,\quad
\forall t\geq t_0.
\end{equation}

Theorem \ref{maintheorem} follows from:
\begin{thm}
\label{maintheorem2} Let $\Phi(t, x)$, $L^\mu(t, x)$,
$h^{\mu\nu}(t,x )$ be given smooth functions satisfying the above
conditions. $A^{\mu\nu}$, $B^{\mu\nu}$ are constants satisfying the
null condition. Assume the initial data $\phi_0(x)$, $\phi_1(x)$ are
smooth and supported in $\{|x|\leq R_0\}$. Then there exists
$\delta_0>0$, depending only on the
  constants $B^{\mu\nu}$, and $\ep_0>0$, depending on
$E_{0}$,$R_0$, $A^{\mu\nu}$, $B^{\mu\nu}$, $\a, t_0, R_1, C_0$, such
that for any $\delta<\delta_0$, $\ep<\ep_0$, there exists a unique
global smooth solution
  $\phi$ of the equation \eqref{THEWAVEEQ} with the property that $\exists$ positive constant $R$,
  depending on $t_0, \a, R_1, C_0, R_0$, such that for the foliation $\Si_\tau$ with radius $R$,
  the solution $\phi$ satisfies
\begin{itemize}
\item[(1)] Energy decay
$$E[\phi](\tau)\leq C E_0\ep^2  (1+\tau)^{-1-\f12\a}.
$$
\item[(2)] Pointwise decay
 \begin{align*}
&\qquad\quad|\phi|\leq C\sqrt{E_0}\ep (1+r)^{-1},\\
&\sum\limits_{|\b|\leq 2}|\pa^\b\phi|\leq C \sqrt{E_0}\ep
(1+r)^{-\f12}(1+|t-r+R|)^{-\f12-\frac{\a}{4}},
\end{align*}
\end{itemize}
where $C$ depends on $R$, $\a_0$, $\a, t_0, R_1, C_0$.
\end{thm}

\bigskip

\begin{remark}
Notice that $\a$ can be arbitrarily small. The decay assumptions on
$L^\mu(t, x)$( condition \eqref{Lcond1}) and $h^{\mu\nu}(t, x)$ are
sharp in the sense that there exists soliton solution( independent
of time t) to the linear wave equation if $L^\mu(t, x)$ behaves like
$(1+|x|)^{-1}$ and any nontrivial $C^3$ solution of the equation
$$\Box \phi=\phi_t^2$$
with compactly supported initial data blows up in finite time
\cite{johnblowup}.
\end{remark}
\begin{remark}
We can also consider equation \eqref{THEWAVEEQ} with zeroth order
linear term $L_0(t, x) \phi$, leading to the same conclusion
provided that $L_0(t, x)$ decays like $(1+|x|)^{-3-a}$. Hence for
the stability problem of large solution( Theorem \ref{maintheorem}),
specific dependence on $w$ of the nonlinearity $\mathcal{N}(w)$ is
also allowed.
\end{remark}

\begin{remark}
For simplicity, we consider the equations in Minkowski space.
However, as in \cite{yang1}, the same conclusion holds on curved
background $(\mathbb{R}^{3+1}, g)$ with metric $g$ merely $C^1$
close to the Minkowski metric and coinciding with the Minkowski
metric outside the cylinder $\{(t, x)||x|\leq R\}$.
\end{remark}
\begin{remark}
 It is not necessary to require that the initial data have compact support. The general assumption
  on the initial data can be that the following quantity
\[
 \sum\limits_{|k|\leq 4}\int_{\mathbb{R}^{3}}r^{1+\a}|\pa Z^k\tilde\phi(0, x)|^2dx,
 \quad \tilde\phi(0, x)=\phi_0(x), \quad\pa_t\tilde\phi(0,
 x)=\phi_1(x),
\]
is finite.
\end{remark}

The small data global existence result of nonlinear wave equations
satisfying the null condition in Minkowski space was first obtained
 by D. Christodoulou \cite{ChDNull} and S. Klainerman
\cite{klnull}. The approach of \cite{klnull} used the vector field
method, introduced by S. Klainerman in \cite{klinvar}. Various
applications of the vector field method to nonlinear wave equations
could be found in ~\cite{klmulti}, \cite{gx-lindblad2},
~\cite{sogge-metcalfe-nakamura}, ~\cite{sogge-metcalfe},
~\cite{sideris-multispeed}, ~\cite{soggemulti}
 . In particular, the celebrated global nonlinear stability of Minkowski space
  has been proven by Christodoulou-Klainerman \cite{kcg} and later by Lindblad-Rodnianski \cite{SMigor}.

\bigskip

The main difficulty of considering nonlinear wave equation
\eqref{THEWAVEEQ} with linear terms $B^{\mu\nu}\pa_\mu\Phi(t,
x)\pa_\nu\phi$, $L^\mu(t, x)\pa_\mu\phi$ and quadratic terms
$h^{\mu\nu}(t, x)\pa_\mu\phi\pa_\nu\phi$ is the rather weak decay of
the functions $\Phi(t, x)$, $L^\mu(t, x)$, $h^{\mu\nu}(t, x)$.
Previous works have relied on the fact that these functions decay to
zero uniformly in time $t$, which is not necessary in this context.
In fact, we even allow these functions to stay static( independent
of $t$) in the cylinder $\{(t, x)||x|\leq R_1\}$. Although we
require $\delta$ to be sufficiently small, which is the only
smallness assumption here, $\delta$ depends only on the constants
$B^{\mu\nu}$.

\bigskip

Our argument here is similar to that in ~\cite{yang1}, which relies
on a new approach, developed by M. Dafermos and I. Rodnianski in
~\cite{newapp}. This new approach is a combination of an integrated
local energy
  inequality and a p-weighted energy inequality in a neighborhood of
 the null infinity. However, due to the weak decay
of $\pa\Phi$, we are not able to obtain the integrated local energy
inequality and the p-weighted energy inequalities separately as in
~\cite{yang1}. We thus consider these two inequalities together, see
Proposition \ref{mainprop} in Section 2 for details.

\bigskip

The plan of this paper is as follows: we establish an integrated
energy inequality in the whole space
 time and two p-weighted energy
 inequalities in Section 2. In Section 3, we use Proposition
 ~\ref{mainprop} to obtain the decay of the
 energy as well as the pointwise decay of the solution under appropriate boostrap assumptions;
 in the last two sections, we close
  our boostrap argument and conclude our main theorems.

\textbf{Acknowlegements} The author is deeply indebted his advisor Igor Rodnianski for
suggesting this problem. He thanks Igor Rodnianski for sharing
numerous valuable thoughts as well as many helpful comments on the
manuscript.

\section{Notations and Preliminaries}
In Minkowski space, we recall the energy-momentum tensor
\[
{\mathbb T}_{\mu\nu}[\phi]=\pa_\mu\phi\pa_\nu\phi-\frac12 g_{\mu\nu}\pa^{\gamma}\phi\pa_{\gamma}\phi.
\]
Given a vector field $X$, we define the currents
\[
J^X_\mu[\phi]= {\mathbb T}_{\mu\nu}[\phi]X^\nu, \qquad
K^X[\phi]= {\mathbb T}^{\mu\nu}[\phi]\pi^X_{\mu\nu},
\]
where $\pi^X_{\mu\nu}=\frac12 \mathcal{L}_Xg_{\mu\nu}$ is the deformation tensor of the vector field $X$.
Recall that
\[
D^\mu J^X_\mu[\phi] = X(\phi)\Box\phi+K^X[\phi].
\]
Let $n$ be the unit normal vector field to hypersurfaces, $d\si$ the
induced surface measure. We denote $d\vol$ as the volume form in
Minkowski space. In null coordinates, we define the null infinity
from $\tau_1$ to $\tau_2$ as follows
\begin{equation*}
\mathcal I_{\tau_1}^{\tau_2} :=\{(u,v,\omega)|u_{\tau_1}\leq u \leq
u_{\tau_2}, v=\infty\}.
\end{equation*}
The corresponding energy flux is
$$I[\phi]_{\tau_1}^{\tau_2}:=\left.\int_{\mathcal I_{\tau_1}^{\tau_2}}
\left((\pa_u\phi)^2+|\nabb\phi|^2\right)
r^2dud\om\right|_{v=\infty},
$$
which is interpreted as a limit when $v\rightarrow\infty$. Denote
$$\tilde{E}[\phi](\tau)=E[\phi](\tau)+I[\phi]_{0}^{\tau}.
$$
Here $E[\phi](\tau)$ is defined in the introduction with the foliation $\Si_\tau$.

\bigskip

Taking a vector field $$X=f(r)\pa_r,$$ for some function $f(r)$,
consider the region bounded by the hypersurfaces $\Si_{\tau_1}$ and
$\Si_{\tau_2}$. By Stoke's formula,  we have the identity
\begin{align}
\notag&\quad\int_{{\Sigma}_{\tau_1}}J^X_\mu[\phi]n^\mu d\sigma - \int_{{\Sigma}_{\tau_2}}J^X_\mu[\phi]n^\mu d\sigma-\int_{\mathcal I_{\tau_1}^{\tau_2}}J^X_\mu[\phi]n^\mu d\sigma\\
&=\int_{\tau_1}^{\tau_2}\int_{\Sigma_\tau}D^\mu
J^X_{\mu}[\phi]d\vol=\int_{\tau_1}^{\tau_2}
\int_{\Sigma_\tau}\Box\phi \cdot X(\phi) + K^X[\phi] d\vol,
\label{energyeq}
\end{align}
where
\begin{align*}
&K^X[\phi]=\mathbb{T}^{\mu\nu}[\phi]\pi^X_{\mu\nu}=(\f12 f'+r^{-1}f)
(\pa_t\phi)^2+(\f12 f'-r^{-1}f)(\pa_r\phi)^2-\f12 f'|\nabb\phi|^2.
\end{align*}
We use $f'$ to denote $\pa_r f$ throughout this paper.

Choose another function $\chi$ of $r$, we have the equality
\begin{align*}
 -\chi\pa^{\gamma}\phi\pa_{\gamma}\phi + \f12\Box\chi\cdot\phi^2
 =\f12 D^{\mu}\left(\pa_{\mu}\chi\cdot \phi^2 - \chi\pa_{\mu}\phi^2\right) +
 \chi\phi\Box\phi.
\end{align*}
Add the above equality to both sides of \eqref{energyeq}. Define the
current
\begin{equation*}
\label{mcurent} \tilde{J}_{\mu}^X[\phi]=J_{\mu}^X[\phi] -
\f12\pa_{\mu}\chi \cdot\phi^2 + \f12 \chi\pa_{\mu}\phi^2.
\end{equation*}
Then we get
\begin{align}
\label{menergyeq}
&\int_{{\Sigma}_{\tau_1}}\tilde{J}^X_\mu[\phi]n^\mu d\sigma - \int_{{\Sigma}_{\tau_2}}\tilde{J}^X_\mu[\phi]n^\mu d\sigma-\int_{\mathcal I_{\tau_1}^{\tau_2}}\tilde{J}^X_\mu[\phi]n^\mu d\sigma\\
\notag&=\int_{\tau_1}^{\tau_2}\int_{\Sigma_\tau}
(f\pa_r\phi+\phi\chi)\Box \phi+ (r^{-1}f + \f12 f' - \chi)(\pa_t\phi)^2\\
\notag&\qquad +(\chi-r^{-1}f+\f12 f')(\pa_r\phi)^2+(\chi-\f12
f')|\nabb\phi|^2- \f12\Box\chi\cdot\phi^2.
\end{align}
For $X=T=\pa_t$ in ~\eqref{energyeq}, we have the energy inequality
\begin{equation}
 \label{energyeqT}
\tilde{E}[\phi](\tau_2)\leq
\tilde{E}[\phi](\tau_1)+2\int_{\tau_1}^{\tau_2}\int_{\Si_\tau}|\Box
\phi||\pa_t\phi|d\vol.
\end{equation}
We now fix the radius $R$ of the foliation $\Si_\tau$. First under
the assumptions in Theorem \ref{maintheorem2}, we choose new
positive constants $\a'$, $t_0'$, $R_1'$ such that
\begin{align}
\label{impsmall}
 &|\pa_v Z^k\Phi|\leq \delta\a' (1+|x|)^{-1-3\a'}, \quad t\geq t'_0,\quad |x|\geq
 R_1',\quad |k|\leq 4,\\
 \label{impcond4}
 &|\pa Z^k\Phi|\leq 2\delta \a'(1+|x|)^{-1-\a'}, \quad t\geq t_0', \quad |x|\leq
 R_1',\quad |k|\leq 4,\\
 \label{t0}
 & (1+t_0')^{\a'} \delta \a'\geq C_0.
\end{align}
In face, since $Z^k\Phi$ is $(\delta, \a, t_0, R_1, C_0)$-weak wave
$\forall |k|\leq 4$, choose $\a'=\f12 \a$ and $R_1'$ large enough
such that \eqref{impsmall} holds. Then \eqref{impcond4} and
\eqref{t0} are satisfied if $t_0'$ is sufficiently large. The other
conditions are also satisfied for these new constants $\a'$, $t_0'$,
$R_1'$. Then the radius $R$ can be fixed as
\[
R=t_0'+R_0,
\]
where $R_0$ is the radius of the support of the initial data. To
avoid too many constants, we still use the constants $\a$, $R_1$,
$t_0$ to denote $\a'$, $R_1'$, $t_0'$ respectively in the sequel.

The following several lemmas, which have been proven in
\cite{yang1}, will be used later on.
\begin{lem}
\label{lem1} On $S_\tau\cup\mathcal I_{0}^{\tau}$
\begin{equation*}
\label{sphereinbd} r\int_{\om}|\phi|^2 d\om\leq
\tilde{E}[\phi](\tau).
\end{equation*}
\end{lem}

\begin{lem}
\label{lem2} If $\phi$ is smooth, then
\begin{equation*}
\label{phiboundH} \int_{r\leq
R}\left(\frac{\phi}{1+r}\right)^2dx+\int_{S_{\tau}}\left(\frac{\phi}{1+r}\right)^2r^2dvd\om\leq
6\tilde{E}[\phi](\tau).
\end{equation*}
\end{lem}

\begin{cor}
 \label{cor1}
In the exterior region $r\geq R$
\begin{equation*}
\label{phipsieq}
\left|\int_{S_\tau}|\pa_v(r\phi)|^2dvd\om-\int_{S_\tau}\phi_v^2
\quad r^2dvd\om \right|\leq 2\tilde{E}[\phi](\tau).
\end{equation*}
\end{cor}

\begin{lem}
\label{prop1} Suppose $f$ and $\chi$ satisfy $$|f|\leq C_1,
\quad|\chi|\leq \frac{C_1}{1+r}, \quad|\chi'|\leq
\frac{C_1}{(1+r)^{2}}$$ for some constant $C_1$, then
\begin{equation*}
\label{bdrycontrol}
\left|\int_{\Sigma_\tau}\tilde{J}_{\mu}^X[\phi]n^{\mu}
d\si\right|\leq 6C_1\tilde{E}[\phi](\tau).
\end{equation*}
\end{lem}

\begin{remark}
 If $\tilde{E}[\phi](\tau)$ is finite, all the above statements are also valid if we replace
 $\tilde{E}[\phi](\tau)$ with $E[\phi](\tau)$.
\end{remark}

\bigskip

Finally, we denote
\begin{equation*}
 \begin{split}
  &\bar\pa\phi:=(\pa\phi, \frac{\phi}{1+r}), \qquad \overline{\pa_v}\phi:=(\pa_v\phi, \nabb\phi),\\
& g(p, \tau):=\int_{S_\tau}r^{p}|\pa_v\psi|^2dvd\om,
\qquad \bar g(p, \tau):= \int_{S_\tau}r^{p}|\overline{\pa_v}\psi|^2dvd\om,\\
 &G[\b, p]_{\tau_1}^{\tau_2}:=\int_{\tau_1}^{\tau_2}(1+\tau)^{-\b}g(p, \tau)d\tau,
 \qquad \bar G[\b, p]_{\tau_1}^{\tau_2}:=\int_{\tau_1}^{\tau_2}(1+\tau)^{-\b}\bar g(p, \tau)d\tau,\\
&D^\b[F]_{\tau_1}^{\tau_2}:=\int_{\tau_1}^{\tau_2}\int_{\Si_\tau}|F|^2(1+r)^{\b+1}d\vol
\end{split}
\end{equation*}
for $\forall p\geq0,\quad \b\in \mathbb{R}^1$. Here $\psi=r\phi$,
$\pa_v=\pa_t+\pa_r$ and $\bar{\pa}\phi$ is a four dimensional vector
with norm $|\bar{\pa}\phi|^2=|\pa\phi|^2+\frac{\phi^2}{(1+r)^2}$,
similarly for $\overline{\pa_v}\phi$. Throughout this paper, we use
the notation $A\les B$ for the inequality $A\leq C B$ with some
constant $C$, depending on $R$, $A^{\mu\nu}$, $B^{\mu\nu}$, $\a$,
$t_0$, $R_1$, $C_0$.

\section{Weighted Energy Inequalities}
In this section, we use the multiplier method to establish an
integrated local energy inequality and two p-weighted energy
inequalities. The integrated local energy inequality was first
proven by C. S. Morawetz in \cite{mora2}. We follow the method in
\cite{dr3} to obtain the integrated local energy inequality here. In
\cite{newapp}, M. Dafermos and I. Rodnianski introduced the
p-weighted energy inequalities in a neighborhood of null infinity.
These two estimates, which the decay of the energy flux
$E[\phi](\tau)$ relies on, were shown separately in ~\cite{yang1}.
Due to the weak decay of the functions $\Phi(t, x)$, $L^\mu(t, x)$,
we are not able to show these two estimates separately. We hence
consider them together.

\bigskip

Consider the following linear wave equation
\begin{equation}
 \label{LWAVEEQ}
\begin{cases}
\Box\phi+N(\Phi, \phi)+L(\pa\phi)=F,\\
\phi(0, x)=\ep \phi_1(x), \quad\pa_t\phi(0, x)=\ep\phi_1(x),
\end{cases}
\end{equation}
where $ N(\Phi, \phi)=B^{\mu\nu}\pa_\mu\Phi\cdot\pa_\nu\phi$,
$L(\pa\phi)=L^\mu(t,x) \pa_\mu\phi$. We have the following key
estimates.
\begin{prop}
\label{mainprop} Suppose that $\Phi$ is a $(\delta, \a, t_0, R_1,
C_0)$-weak wave for positive constants $\delta, \a, t_0, R_1, C_0$.
Assume the given functions $L^\mu(t,x)$ satisfy
\begin{align*}
|L^\mu(t, x)|\leq C_0,\quad t\leq t_0
\end{align*}
and \textbf{one} of the following two conditions
\[
 |L^\mu(t, x)|\leq \delta \a(1+|x|)^{-1-3\a},\quad t\geq t_0,
\]
or
\[
 |L^\mu (t, x)|\leq C_0
(1+|x|)^{-1-3\a}(1+(t-|x|)_{+})^{-\a},\quad \forall t\geq t_0.
\]
Suppose the constants $\a, t_0, \delta,C_0$ obey the relation
\eqref{t0}. Then there exists $\delta_0>0$, depending only on the
constants $B^{\mu\nu}$, such that for all $ \delta<\delta_0$,
solution $\phi$ of equation \eqref{LWAVEEQ} has the following
properties:
\begin{itemize}
\item[(1)] Integrated local energy estimate
\begin{equation}
 \label{ILE0}
\begin{split}
\int_{\tau_1}^{\tau_2}\int_{\Si_\tau}\frac{|\bar\pa\phi|^2}{(1+r)^{\a+1}}dxd\tau &\les
E[\phi](\tau_1)+D^{\a}[F]_{\tau_1}^{\tau_2}+(1+\tau_1)^{-\a}D^{2\a}[F]_{\tau_1}^{\tau_2}\\
&+(1+\tau_1)^{-1-2\a}\left(g(1+2\a,
\tau_2)+\int_{\tau_1}^{\tau_2}\tau^{\f12\a}D^{2\a}[F]_{\tau}^{\tau_2}d\tau\right).
\end{split}
\end{equation}
\item[(2)] Energy bound
\begin{equation}
 \label{eb}
\begin{split}
E[\phi](\tau_2)&\les E[\phi](\tau_1)+D^{\a}[F]_{\tau_1}^{\tau_2}+(1+\tau_1)^{-\a}D^{2\a}[F]_{\tau_1}^{\tau_2}\\
&+(1+\tau_1)^{-1-2\a}\left(g(1+2\a,
\tau_2)+\int_{\tau_1}^{\tau_2}\tau^{\f12\a}D^{2\a}[F]_{\tau}^{\tau_2}d\tau\right).
\end{split}
\end{equation}
\item[(3)] p-weighted energy inequalities in a neighborhood of null infinity
\begin{align}
\notag
  g(1, \tau_2)+ \int_{\tau_1}^{\tau_2}E[\phi](\tau)d\tau&\les g(1, \tau_1)+\tau_1^{1-\a}
   D^{2\a}[F]_{\tau_1}^{\tau_2}+\tau_1^{1-\a}E[\phi](\tau_1)+\int_{\tau_1}^{\tau_2}(1+\tau)^{-\a}D^{2\a}
   [F]_{\tau_1}^{\tau_2}\\
\label{pWE1}
&+(1+\tau_1)^{-2\a}\int_{\tau_1}^{\tau_2}\tau^{\f12\a}D^{2\a}[F]_{\tau}^{\tau_2}d\tau+(1+\tau_1)^{-2\a}g(1+2\a, \tau_1),\\
\notag
  g(1+2\a, \tau_2)+ \bar G[0, 2\a]_{\tau_1}^{\tau_2}&\les g(1+2\a, \tau_1)+(1+\tau_1)^{1-\a}E[\phi](\tau_1) \\
\label{pWE1a}
&+\int_{\tau_1}^{\tau_2}\tau^{\f12\a}D^{2\a}[F]_{\tau}^{\tau_2}d\tau+(1+\tau_1)^{1+\f12\a}
D^{2\a}[F]_{\tau_1}^{\tau_2}.
\end{align}
\end{itemize}
\end{prop}

\begin{remark}
We mention here that variants and generalizations of estimate \eqref{ILE0} can also be found
in \cite{sogge-metcalfe2}, \cite{sogge-metcalfe}.
\end{remark}

The following corollary will be used to derive the energy decay
estimates when commuting the equation with the vector fields $Z$.
\begin{cor}
 \label{D2aQp}
Assume the given functions $\Phi(t, x)$, $L^\mu(t, x)$ and the
constant $\delta$ satisfy the conditions in the above proposition.
Then for solution $\phi$ of ~\eqref{LWAVEEQ},
 we have estimates for $N=B^{\mu\nu}\pa_\mu\Phi\cdot\pa_\nu\phi$,
$L=L^\mu\pa_\mu\phi$
\begin{equation*}
 \label{D2aQ}
\begin{split}
D^{2\a}[N]_{\tau_1}^{\tau_2}+D^{2\a}[L]_{\tau_1}^{\tau_2}&\les E[\phi](\tau_1)+D^{\a}[F]_{\tau_1}^{\tau_2}+(1+\tau_1)^{-\a}D^{2\a}[F]_{\tau_1}^{\tau_2}\\
&\qquad+(1+\tau_1)^{-1-2\a}\left(g(1+2\a,
\tau_2)+\int_{\tau_1}^{\tau_2}\tau^{\f12\a}D^{2\a}[F]_{\tau}^{\tau_2}d\tau\right).
\end{split}
\end{equation*}
\end{cor}

Under appropriate boostrap assumptions on the nonlinearity $F$, the
above inequalities lead to
 decay of the energy flux $E[\phi](\tau)$. We discuss the integrated local energy inequality and the
 p-weighted energy inequalities separately. And then combine them together to prove the above proposition. The following
  two lemmas will be used frequently. First define
  \[
  A=10\sup\limits_{\mu, \nu} \{|B^{\mu\nu}|\}.
  \]

\begin{lem}
 \label{lnullQ}
Let $N=B^{\mu\nu}\pa_\mu\Phi\pa_\nu\phi$. Then
\begin{equation*}
 \label{nullQ}
|rN|\leq
A\left(|\pa\Phi||\overline{\pa_v}\psi|+|\pa_v\Phi||\pa\psi|+|\pa\Phi||\phi|\right),\quad
\psi=r\phi.
\end{equation*}
\end{lem}
\begin{proof}
 By our notations
\[
 r N=rB^{\mu\nu}\pa_\mu\Phi\cdot\pa_\nu\phi=B^{\mu\nu}\pa_\mu\Phi\cdot\pa_\nu\psi-B^{\mu\nu}\pa_\mu\Phi\pa_\nu
 r\cdot\phi.
\]
The lemma then follows from the fact that $B^{\mu\nu}$ satisfies the
null condition and the inequality $|\pa r|\leq 1$.
\end{proof}

\begin{lem}[Gronwall's Inequality]
\label{lGronwall} Suppose $A(\tau)$, $E(\tau)$ are nonnegative
functions on $[\tau_1, \tau_2]$ . Assume that $E(\tau)$ is
nondecreasing on this interval and $\b$ is a positive number. If
\[
 A(\tau)\leq  E[\tau]+C\int_{\tau_1}^{\tau}(1+s)^{-1-\b}A(s)ds, \qquad \forall\tau\in[\tau_1,
 \tau_2],
\]
then
\begin{equation*}
 \label{Gronwall}
A(\tau)\leq \exp\left(C{\b}^{-1}(1+\tau_1)^{-\b}\right)E(\tau),
\qquad \forall\tau\in[\tau_1, \tau_2].
\end{equation*}
\end{lem}
\begin{proof}
 See ~\cite{sogge}.
\end{proof}

\subsection{Integrated Local Energy Inequality}
We follow the idea used in ~\cite{dr3} by choosing appropriate
functions $f$ and $\chi$ such that the coefficients on the right
hand side of \eqref{menergyeq} are positive. The left hand side can
be controlled by the energy flux $\tilde{E}[\phi]$ by Lemma
\ref{prop1}. We thus end up with an integrated energy inequality in
the whole space time. We now discuss this in detail.

Take
$$f=\b-\frac{\beta}{(1+r)^{\alpha}},\quad \chi=r^{-1}f,\quad \b=\frac{2}{\a}.$$
Notice that
\[
 \frac{(1+r)^\a-1}{r}\geq \frac{\a}{1+r}.
\]
We have
\begin{align*}
 &r^{-1}f+\f12 f'-\chi=\chi-r^{-1}f+\f12 f'=\frac{1}{(1+r)^{1+\a}},\\
&\chi-\f12 f'=\frac{\b\left((1+r)^\a-1\right)}{r(1+r)^{\a}}-\frac{1}{(1+r)^{1+\a}}\geq\frac{1}{(1+r)^{\a+1}},\\
&-\f12\Box \chi=\frac{\a+1}{r(1+r)^{2+\a}}.
\end{align*}
Hen the energy inequalities ~\eqref{menergyeq}, ~\eqref{energyeqT}
together with Lemma \ref{prop1} imply that
\begin{equation}
 \label{ILEQ}
\int_{\tau_1}^{\tau_2}\int_{\Si_\tau}\frac{|\bar\pa\phi|^2}{(1+r)^{1+\a}}d\vol\leq
12\b\tilde{E}[\phi](\tau_1)
+13\b\int_{\tau_1}^{\tau_2}\int_{\Si_\tau}|F-N-L||\bar\pa\phi|d\vol.
\end{equation}
To proceed, we have to estimate the linear terms $N(\Phi, \phi)$,
$L(\pa\phi)$. We first consider the case $\tau_2\geq\tau_1\geq t_0$.
For $\tau\geq t_0$, notice that on $\Si_\tau$
\[
C_0(1+|x|)^{-\a}(1+(t-|x|)_{+})^{-\a}\leq C_0(1+\tau)^{-\a}\leq
C_0(1+t_0)^{-\a}\leq\delta \a
\]
by the inequality \eqref{t0}(we have assumed this inequality in
Proposition \ref{mainprop}). Hence under the conditions on the
functions $L^{\mu}(t, x)$ in Proposition \ref{mainprop}, we always
have
\[
\int_{\tau_1}^{\tau_2}\int_{\Si_\tau}|L(\pa\phi)||\bar\pa\phi|d\vol\leq
\delta\a\int_{\tau_1}^{\tau_2}\int_{\Si_\tau}(1+r)^{-1-\a}|\bar\pa\phi|^2d\vol.
\]
For $N(\Phi,\phi)$, we consider it inside and outside the cylinder
$\{|x|\leq R_1\}$ separately. When $r=|x|\leq R_1$, the null
structure of $N(\Phi, \phi)$ is not necessary. We has to rely on the
smallness of $\delta$. Since $\Phi$ is a weak wave, condition $(iv)$
of Definition \ref{weakwave} implies that
\begin{equation*}
 \label{Qin}
|N|=|B^{\a\b}\pa_\a\Phi\pa_\b\phi|\leq A  \delta\a
(1+r)^{-1-\a}|\pa\phi|.
\end{equation*}
For $r\geq R_1$, the null structure of $N(\Phi, \phi)$ is of
particular importance. By Lemma \ref{lnullQ}, it suffices to
estimate the three terms $r^{-1}|\pa
\Phi||\overline{\pa_v}\psi||\bar\pa\phi|$, $r^{-1}|\pa_v\Phi||\pa
\psi||\bar\pa\phi|$, $r^{-1}|\pa\Phi||\phi||\bar\pa\phi|$. Without
loss of generality, assume $R_1\geq 1$. For the second term,
inequality \eqref{impsmall} shows that
\begin{equation*}
 \label{pavPhipsi}
|\pa_v\Phi||r^{-1}\pa\psi||\bar\pa\phi|\leq 2\delta
\a(1+r)^{-1-\a}|\bar\pa\phi|^2.
\end{equation*}
On $\Si_\tau\cap \{|x|\geq R_1\geq 1\}$, for the first term,  we
have
\begin{align*}
 \notag
r^{-1}|\pa\Phi||\overline{\pa_v}\psi||\bar\pa\phi|&\leq C(1+r)^{-\frac{3}{2}}(1+\tau)^{-\f12-4\a}|\overline{\pa_v}\psi||\bar\pa\phi|\\
\label{paPhipavpsi}
 &\leq \delta\a(1+r)^{-1-\a}|\bar\pa\phi|^2+C(1+\tau)^{-1-8\a}r^{-2+\a}|\overline{\pa_v}
 \psi|^2.
\end{align*}
Here we denote $C$ as a constant depending on $\a$, $R=t_0+R_0$,
$A^{\a\b}$, $B^{\a\b}$, $C_0$, $\delta$. Similarly for the third
term, we have
\begin{equation*}
 \label{paPhiphi}
r^{-1}|\pa\Phi||\phi||\bar\pa\phi|\leq
\delta\a(1+r)^{-1-\a}|\bar\pa\phi|^2+C(1+\tau)^{-1-8\a}r^{-2+\a}|\phi|^2.
\end{equation*}
It remains to control $r^{-2+\a}|\phi|^2$. We use the Hardy's
inequality outside the cylinder $\{|x|\leq R\}$. By Lemma
~\ref{sphereinbd}, we have
\begin{equation}
\label{phi2bd}
\begin{split}
\int_{\om}|\psi|^2(\tau,v,\om)d\om&\leq C\int_{\om}|\psi|^2(\tau, v_\tau, \om)d\om +C\left(\int_{v_\tau}^v\int_{\om}|\pa_v\psi|d\om dv\right)^2\\
 &\leq C\tilde{E}[\phi](\tau) +  C\int_{v_\tau}^v\int_{\om}r^{1+2\a}|\pa_v\psi|^2d\om dv\int_{v_\tau}^v r^{-1-2\a}dv\\
&\leq C\tilde{E}[\phi](\tau)+ C g(1+2\a, \tau),\quad (\tau, v,
\om)\in S_\tau,
\end{split}
\end{equation}
 where $v_\tau=\frac{R+\tau}{2}$, $v=\frac{r+t}{2}$. Hence for all
$p\leq 1+2\a$
\begin{equation}
 \label{phi21a}
\int_{S_\tau}r^{p-3\a}\phi^2dvd\om=\int_{v_\tau}^{\infty}r^{p-2-3\a}\int_{\om}|\psi|^2d\om
dv \leq C\tilde{E}[\phi](\tau)+ C g(1+2\a, \tau).
\end{equation}
On the other hand, Lemma ~\ref{lem2} shows that
\begin{equation*}
 \int_{S_\tau}\phi^2 dvd\om \leq C \tilde{E}[\phi](\tau).
\end{equation*}
Interpolate with \eqref{phi21a} for  $p=1+2\a$. We derive
\begin{equation*}
 \label{phi2a}
\begin{split}
\int_{S_\tau}r^\a\phi^2dvd\om&\leq C \tilde{E}[\phi](\tau)^{1-\gamma}\left(\tilde{E}[\phi](\tau)+g(1+2\a, \tau)\right)^{\gamma}\\
&\leq
C\tilde{E}[\phi](\tau)+C\tilde{E}[\phi](\tau)^{1-\gamma}g(1+2\a,
\tau)^{\gamma},
\end{split}
\end{equation*}
where $\gamma=\frac{\a}{1-\a}$. This gives estimates for $\phi^2$
outside the cylinder $\{|x|\leq R \}$.

In the region $R_1\leq r\leq R$, using Sobolev embedding and Lemma
\ref{sphereinbd}, we get
\begin{equation*}
 \label{phi2bdin}
\int_{\om}\phi^2 d\om\leq C
\left.\int_{\om}\phi^2d\om\right|_{r=R}+C\int_{r\leq
R}|\pa_r\phi|^2dx\\leq C\tilde{E}[\phi](\tau).
\end{equation*}
Therefore we can estimate $r^{-2+\a}|\phi|^2$ outside the cylinder
$\{|x|\leq R_1\}$ as follows
\begin{equation*}
\begin{split}
 \int_{\{r\geq R_1\}\cap \Si_\tau}r^{-2+\a}|\phi|^2d\si&=\int_{R_1\leq r\leq R}r^{-2+\a}\phi^2dx+\int_{S_\tau}r^\a \phi^2dvd\om\\
&\leq
C\tilde{E}[\phi](\tau)+C\tilde{E}[\phi](\tau)^{1-\gamma}g(1+2\a,
\tau)^{\gamma}.
\end{split}
\end{equation*}
Inside the cylinder $\{|x|\leq R_1\}$, we use the assumption that
$\pa\Phi$ is small. Hence combining the above estimates, we can
bound the linear term $(|N(\Phi, \phi)|+|L(\pa\phi)|)|\bar\pa \phi|$
in \eqref{ILEQ} as follows
\begin{align*}
 \b \int_{\tau_1}^{\tau_2}\int_{\Si_\tau}(|N|+|L|)|\bar\pa\phi|d\vol&\leq A\delta\int_{\tau_1}^{\tau_2}\int_{\Si_\tau}\frac{|\bar\pa\phi|^2}{(1+r)^{1+\a}}d\vol+C\int_{\tau_1}^{\tau_2}\frac{\tilde{E}[\phi](\tau)}{(1+\tau)^{1+8\a}}d\tau+ C \bar G[1+8\a, \a]_{\tau_1}^{\tau_2}\\
&+C\left(\int_{\tau_1}^{\tau_2}(1+\tau)^{-1-2\a}\tilde{E}[\phi](\tau)d\tau\right)^{1-\gamma}\left(\int_{\tau_1}^{\tau_2}(1+\tau)^{-7+4\a}g(1+2\a,\tau) d\tau\right)^{\gamma}\\
&\leq A\delta\int_{\tau_1}^{\tau_2}\int_{\Si_\tau}\frac{|\bar\pa\phi|^2}{(1+r)^{1+\a}}d\vol+C\int_{\tau_1}^{\tau_2}\frac{\tilde{E}[\phi](\tau)}{(1+\tau)^{1+2\a}}d\tau\\
&\qquad+ C G[2+2\a, 1+2\a]_{\tau_1}^{\tau_2}+C \bar G[1+2\a,
\a]_{\tau_1}^{\tau_2},
\end{align*}
where we used H$\ddot{o}$lder's inequality and Jensen's inequality
\[
 a^{1-\gamma}b^{\gamma}\leq (1-\gamma)a +\gamma b,\quad \forall a,
 b>0.
\]
For the inhomogeneous term $|F||\bar \pa \phi|$ in \eqref{ILEQ}, we
have
\begin{equation*}
 \int_{\tau_1}^{\tau_2}\int_{\Si_\tau}|F||\bar\pa\phi|d\vol\leq \delta\a \int_{\tau_1}^{\tau_2}\int_{\Si_\tau}\frac{|\bar\pa\phi|^2}{(1+r)^{1+\a}}d\vol+ C D^\a[F]_{\tau_1}^{\tau_2}
\end{equation*}
If we choose
$$\delta_0=\frac{A}{100},$$
then for all $\delta<\delta_0$, inequality \eqref{ILEQ} implies that
\begin{equation*}
\begin{split}
 \int_{\tau_1}^{\tau_2}\int_{\Si_\tau}\frac{|\bar\pa\phi|^2}{(1+r)^{1+\a}}d\vol&\les \tilde{E}[\phi](\tau_1)+ D^\a[F]_{\tau_1}^{\tau_2}+\int_{\tau_1}^{\tau_2}\frac{\tilde{E}[\phi](\tau)}{(1+\tau)^{1+2\a}}d\tau\\
&+  G[2+2\a, 1+2\a]_{\tau_1}^{\tau_2}+ \bar G[1+2\a,
\a]_{\tau_1}^{\tau_2}.
\end{split}
\end{equation*}
Similarly, the energy inequality ~\eqref{energyeqT} shows that
\begin{equation*}
\begin{split}
 \tilde{E}[\phi](\tau_2)&\les \tilde{E}[\phi](\tau_1)+\int_{\tau_1}^{\tau_2}
 \frac{\tilde{E}[\phi](\tau)}{(1+\tau)^{1+2\a}}d\tau+  G[2+2\a, 1+2\a]_{\tau_1}^{\tau_2}+
 \bar G[1+2\a, \a]_{\tau_1}^{\tau_2}+D^\a[F]_{\tau_1}^{\tau_2}.
\end{split}
\end{equation*}
We Gronwall's inequality to control the second term on the right
hand side of the above inequality. We thus have
\begin{equation}
 \label{EING}
\tilde{E}[\phi](\tau_2)\les \tilde{E}[\phi](\tau_1)+  G[2+2\a,
1+2\a]_{\tau_1}^{\tau_2}+ \bar G[1+2\a,
\a]_{\tau_1}^{\tau_2}+D^\a[F]_{\tau_1}^{\tau_2}.
\end{equation}
Then the above integrated local energy inequality is improved to
\begin{equation}
\label{ILEG}
\begin{split}
 \int_{\tau_1}^{\tau_2}\int_{\Si_\tau}\frac{|\bar\pa\phi|^2}{(1+r)^{1+\a}}d\vol&\les
 \tilde{E}[\phi](\tau_1)+ D^\a[F]_{\tau_1}^{\tau_2}+  G[2+2\a, 1+2\a]_{\tau_1}^{\tau_2}+
 \bar G[1+2\a, \a]_{\tau_1}^{\tau_2}.
\end{split}
\end{equation}
We have shown \eqref{EING}, \eqref{ILEG} for all $\tau_2\geq
\tau_1\geq t_0$. We claim that these two inequalities hold for all
$\tau_2\geq \tau_1\geq 0$. In fact, when $\tau_1\leq \tau_2\leq
t_0$, the finite speed of propagation for wave equation
~\cite{sogge} shows that $\phi$ vanishes when $r\geq R=t_0+R_0$.
Hence we can show
\[
 \int_{\tau_1}^{\tau_2}\int_{\Si_\tau}(|N|+|L|)|\bar\pa\phi|dxd\tau\les \int_{\tau_1}^{\tau_2}
 \int_{r\leq R}|\bar\pa\phi|^2dxd\tau\les
 \int_{\tau_1}^{\tau_2}\tilde{E}[\phi](\tau)d\tau,
\]
When considering the energy inequality \eqref{energyeqT},
$\int_{\tau_1}^{\tau_2}\tilde{E}[\phi](\tau)d\tau$ can be absorbed
by using Gronwall's inequality since $\tau_2\leq t_0$. Hence we can
conclude ~\eqref{EING}, ~\eqref{ILEG} for all $0\leq\tau_1\leq
\tau_2\leq t_0$. For the case $\tau_1\leq t_0\leq \tau_2$, split the
interval $[\tau_1, \tau_2]$ into $[\tau_1, t_0]$ and $[t_0,
\tau_2]$, on which we have two separate inequalities. Combining them
together, we get ~\eqref{EING}, ~\eqref{ILEG}. Therefore
\eqref{EING}, \eqref{ILEG} hold for all $0\leq \tau_1\leq \tau_2$.

\bigskip

We end this section by making a remark. We have used the modified
energy flux $\tilde{E}[\phi](\tau)$ instead of $E[\phi](\tau)$ to
make the above argument rigorous.
 We claim that the inequalities \eqref{EING}, ~\eqref{ILEG} hold if we
 replace $\tilde{E}[\phi](\tau)$ with $E[\phi](\tau)$. In fact, it
 is sufficient to consider the case when
\[
\tilde{E}[\phi](\tau_1)+D^\a[F]_{\tau_1}^{\tau_2}+  G[2+2\a,
1+2\a]_{\tau_1}^{\tau_2}+ \bar G[1+2\a, \a]_{\tau_1}^{\tau_2}
\]
is finite. By \eqref{EING}, this shows that
 $\tilde{E}[\phi](\tau)$ is finite for all $\tau\in[\tau_1, \tau_2]$.
  Thus Remark 1 shows that all the above statements hold if we replace
  $\tilde{E}[\phi](\tau)$ with $E[\phi](\tau)$ for $\tau\in[\tau_1,
  \tau_2]$. In the sequel, we no longer use the modified energy flux
  $\tilde{E}[\phi](\tau)$ for the reason argued here.

\subsection{p-weighted Energy inequality}
We revisit the p-weighted energy inequalities developed by M.
Dafermos and I. Rodnianski in ~\cite{newapp}. Rewrite the equation
~\eqref{LWAVEEQ} in null coordinates
\begin{equation}
\label{waveqpsi} -\pa_u \pa_v \psi+\lap \psi=r(F-N-L),\quad
\psi:=r\phi,
\end{equation}
where $\lap$ denotes the Laplacian on the sphere with radius $r$.
Multiplying the equation by $r^p \pa_v\psi$ and integration by parts
in the region bounded by the two null hypersurfaces $S_{\tau_1},
S_{\tau_2}$ and the hypersurface $\{r=R\}$, we obtain
\begin{align}
\notag
&\int_{S_{\tau_2}} r^p (\pa_v\psi)^2 dvd\om +2\int_{\tau_1}^{\tau_2}\int_{S_\tau}r^{p+1}(F-N-L)\pa_v\psi dvd\tau d\om\\
\notag
& +\int_{\tau_1}^{\tau_2}\int_{S_\tau} r^{p-1}  \left (p(\pa_v\psi)^2 +
(2-p) |\nabb\psi|^2\right)dvd\tau d\om +\int_{\mathcal I_{\tau_1}^{\tau_2}} r^p |\nabb\psi|^2 dud\tau d\om\\
\label{pWE} =& \int_{S_{\tau_1}}r^p (\pa_v\psi)^2 dvd\om
+\int_{\tau_1}^{\tau_2} r^p \left (|\nabb\psi|^2-
(\pa_v\psi)^2\right)d\om d\tau |_{r=R}.
\end{align}
We claim that we can estimate the boundary terms on $r=R$ as follows
\begin{equation}
 \begin{split}
  &\left|\int_{\tau_1}^{\tau_2} r^p \left (|\nabb\psi|^2- (\pa_v\psi)^2\right)d\om d\tau \right|_{r=R}\\
&\les E[\phi](\tau_1)+  G[2+2\a, 1+2\a]_{\tau_1}^{\tau_2}+ \bar
G[1+2\a, \a]_{\tau_1}^{\tau_2}+D^\a[F]_{\tau_1}^{\tau_2}.
 \end{split}
\label{bdestR}
\end{equation}
Since $R$ is a fixed constant, it suffices to show \eqref{bdestR}
for $p=0$. Thus take $p=0$ in the identity \eqref{pWE}. The energy
term on the null hypersurfaces $S_{\tau_1}$, $S_{\tau_2}$,
$\mathcal{I}_{\tau_1}^{\tau_2}$ can be bounded by
$\tilde{E}(\tau_2)+\tilde{E}(\tau_1)$, which can be estimated by using the energy
inequality \eqref{EING}. We use the improved integrated local energy estimates for $\nabb\phi$ to bound the third term in \eqref{pWE}.
 Recall that when $r\geq R$, we in fact have the improved lower bound
\[
\frac{1}{r}\les
\frac{\b\left((1+r)^\a-1\right)}{r(1+r)^{\a}}-\frac{1}{(1+r)^{1+\a}}=\chi-\f12
f'
\]
instead of $\frac{1}{(1+r)^{1+\a}}$ we have used in \eqref{menergyeq} to obtain \eqref{ILEQ}.
Thus we actually can show that
\[
\int_{\tau_1}^{\tau_2}\int_{S_\tau} r^{-1}  |\nabb\psi|^2dvd\om
d\tau=\int_{\tau_1}^{\tau_2}\int_{S_\tau}\frac{|\nabb\phi|^2}{r}d\vol\les
E[\phi](\tau_1)+\int_{\tau_1}^{\tau_2}\int_{\Si_\tau}|F-N-L||\bar{\pa}\phi|d\vol.
\]
For the inhomogeneous term, notice that
\[
 \int_{\tau_1}^{\tau_2}\int_{S_\tau}|r(F-N-L)\pa_v\psi|dvd\om d\tau\les \int_{\tau_1}^{\tau_2}\int_{S_\tau}
 |F-N-L||\bar\pa\phi|d\vol.
\]
We have already shown that this term can be bounded by the right hand side of \eqref{bdestR} in the previous section. Thus
the inequality \eqref{bdestR} follows.

\bigskip

Now, to make use of the identity \eqref{pWE}, we need to control the
inhomogeneous term $r^{p+1}(F-N-L)\pa_v\psi$ as
 all the other terms have a positive
sign or are bounded. Due to the different structures of $F$, $N$, $L$, we discuss them separately. The most difficult term
is the linear term $r^{p+1}N(\Phi, \phi)\pa_v\psi$ satisfying the null condition. For this term, by Lemma ~\ref{lnullQ},
it suffices to estimate the following three
terms
\[
r^p|\pa\Phi||\overline{\pa_v}\psi||\pa_v\psi|,\quad
r^p|\pa\Phi||\phi||\pa_v\psi|, \quad r^p|\pa_v\Phi||\pa\psi||\pa_v\psi|.
\]
In application, $p\in (0, 2)$. In particular, the coefficients $p$, $2-p$ in \eqref{pWE} are positive. From
the decay assumptions on $\Phi$( see Definition \ref{weakwave}), we estimate the first term as follows
\begin{equation}
\label{Phibarpsi}
\begin{split}
 2r^p|\pa \Phi||\overline{\pa_v}\psi||\pa_v\psi|&\leq2 r^{p-\f12}(1+(t-|x|)_+)^{-\f12-4\a}|\overline{\pa_v}\psi||\pa_v\psi|\\
&\leq\ep_1
r^{p-1}|\overline{\pa_v}\psi|^2+\frac{C}{\ep_1}r^p(1+\tau)^{-1-8\a}|\pa_v\psi|^2,\quad
\forall \ep_1>0.
\end{split}
\end{equation}
The first term will be absorbed if $p>\ep_1$, $2-p>\ep_1$, while the second term will be controlled by using
Gronwall's inequality. Similarly for the second term $r^p|\pa\Phi||\phi||\pa_v\psi|$, we can show
\begin{equation*}
 2r^p|\pa \Phi||\phi||\pa_v\psi|\les
 r^{p-1+3\a}(1+\tau)^{-6\a}|\pa_v\psi|^2+r^{p-3\a}(1+\tau)^{-1-2\a}|\phi|^2.
\end{equation*}
We use interpolation to further bound the first term on the right hand side of the above inequality. Notice that
\[
 p\cdot \frac{5\a}{1+2\a}\geq p-1+3\a,\quad p\leq 1+2\a.
\]
Using H$\ddot{o}$lder's inequality and Jensen's inequality, we
have
\[
 r^{p-1+3\a}(1+\tau)^{-6\a}\leq \left(r^p(1+\tau)^{-1-\a}\right)^{\frac{5\a}{1+2\a}}\cdot
 \left((1+\tau)^{-\a}\right)^{1-\frac{5\a}{1+2\a}}\leq
 \tau^{-\a}+r^p(1+\tau)^{-1-\a}.
\]
We use estimate \eqref{phi21a} to bound $r^{p-3\a}\phi^2$. Summarizing, we can show that for $p\leq 1+2\a$
\begin{equation}
\label{Phiphip}
\begin{split}
&\int_{\tau_1}^{\tau_2}\int_{S_\tau}r^p|\pa \Phi||\phi||\pa_v\psi|dvd\om d\tau\\
&\les \tau_1^{-\a}G[0, 0]_{\tau_1}^{\tau_2}+G[1+\a, p]_{\tau_1}^{\tau_2}+
 G[1+2\a,
 1+2\a]_{\tau_1}^{\tau_2}+\int_{\tau_1}^{\tau_2}\frac{\tilde{E}[\phi](\tau)}{(1+\tau)^{1+2\a}}d\tau.
\end{split}
\end{equation}
It remains to handle the third term
$r^p|\pa_v\Phi||\pa\psi||\pa_v\psi|$. We estimate this term and the linear term $r^{p+1}L(\pa\phi) \pa_v\psi$ together due to the
similar assumptions on $\pa_v\Phi$, $L^\mu(t, x)$. The difficulty for estimating these two terms is that we are not allowed
 to use Cauchy-Schwartz's inequality as we did previously.
However notice that the integrated local energy is expected to decay in $\tau$( $(1+\tau)^{-1-\a}$). We can put some positive weights
of $\tau$ in the integrated local energy such that it is still bounded. To start with, observe that when $|x|\geq R\geq 1$,
we have
\[
r^p|\pa_v\Phi||\pa\psi|+r^{p+1}|L(\pa\phi)|\les
r^{p-1-3\a}|\pa\psi|+r^{p+1-1-3\a}|\pa\phi|\les r^{p-3\a}|\bar\pa
\phi|.
\]
Thus we can bound
\begin{equation*}
 \begin{split}
  &\int_{\tau_1}^{\tau_2}\int_{S_\tau}r^p|\pa_v \Phi||\pa\psi||\pa_v\psi|+r^{p+1}|L(\pa\phi)||\pa_v\psi|dvd\om d\tau
\les\int_{\tau_1}^{\tau_2}\int_{S_\tau}r^{p-3\a}|\bar\pa\phi||\pa_v\psi|dvd\om d\tau\\
&\les\left(\int_{\tau_1}^{\tau_2}\tau^{1-\a}\int_{S_\tau}\frac{|\bar\pa\phi|^2}{(1+r)^{1+\a}}d\vol\right)^\f12\left(G[1-\a, 2p-1-5\a]_{\tau_1}^{\tau_2}\right)^\f12\\
&\les
\frac{1}{\ep_2}\int_{\tau_1}^{\tau_2}\tau^{1-\a}\int_{S_\tau}\frac{|\bar\pa\phi|^2}{(1+r)^{1+\a}}d\vol+\ep_2
\left(G[0, 0]_{\tau_1}^{\tau_2}\right)^{\frac{2\a}{1+\a}}
 \left(G[1+\a, p]_{\tau_1}^{\tau_2}\right)^{\frac{1-\a}{1+\a}}\\
 &\les \frac{1}{\ep_2}\int_{\tau_1}^{\tau_2}\tau^{1-\a}\int_{S_\tau}\frac{|\bar\pa\phi|^2}{(1+r)^{1+\a}}d\vol+\ep_2
 \left(G[0, 0]_{\tau_1}^{\tau_2}
 +G[1+\a, p]_{\tau_1}^{\tau_2}\right)
 \end{split}
\end{equation*}
for all positive number $\ep_2\leq 1$. Here we have used the fact
\[
 2p-1-5\a-\frac{1-\a}{1+\a}p\leq 0, \quad p\leq 1+2\a.
\]
We now have to show that the first term on the right hand side is bounded. We rely on the following lemma.
\begin{lem}
 \label{lweightILE}
Suppose $f(\tau)$ is smooth. Then for any $\b\neq 0$, we have the identity
\begin{equation*}
\int_{\tau_1}^{\tau_2}s^\b
f(s)ds=\b\int_{\tau_1}^{\tau_2}\tau^{\b-1}\int_{\tau}^{\tau_2}f(s)ds
d\tau+\tau_1^{\b}\int_{\tau_1}^{\tau_2}f(s)ds.
\end{equation*}
\end{lem}
\begin{proof}
 Let
\[
 F(\tau)=\int_{\tau}^{\tau_2}f(s)ds.
\]
Integration by parts gives the lemma.
\end{proof}
Apply the lemma to $\b=1-\a$,
$f(\tau)=\int_{S_\tau}\frac{|\bar\pa\phi|^2}{(1+r)^{1+\a}}d\si$.
Then the integrated local energy inequality \eqref{ILEG} implies
that
\begin{equation*}
\begin{split}
\int_{\tau_1}^{\tau_2}\tau^{1-\a}\int_{S_\tau}\frac{|\bar\pa\phi|^2}{(1+r)^{1+\a}}d\vol&\les \tau_1^{-\a}\int_{\tau_1}^{\tau_2}E[\phi](\tau)d\tau+\tau_1^{1-\a}E[\phi](\tau_1)+\int_{\tau_1}^{\tau_2}\tau^{-\a}D^\a[F]_{\tau}^{\tau_2}d\tau\\
&\qquad+G[1+2\a, 1+2\a]_{\tau_1}^{\tau_2}+\bar G[2\a,
\a]_{\tau_1}^{\tau_2}+\tau_1^{1-\a} D^\a[F]_{\tau_1}^{\tau_2}.
\end{split}
\end{equation*}
Since in application only two p-weighted energy inequalities
associated to $p=1$ and $p=1+2\a$ are considered, we use interpolation to bound $\bar G[2\a, \a]$
\[
 \bar G[2\a, \a]_{\tau_1}^{\tau_2}\les \ep_2\ep_3 \bar G[2\a, 2\a]_{\tau_1}^{\tau_2}+\frac{1}{\ep_2\ep_3}\tau_1^{-2\a}\bar G[0, 0]_{\tau_1}^{\tau_2}
\]
for all positive $\ep_3$, where $\ep_2$ is the constant appeared before.

Our ultimate goal is to derive the decay of the energy flux
$E[\phi](\tau)$ on $\Si_\tau$. The almost energy flux $\bar{g}(0, \tau)$ on $S_\tau$ is related to $E[\phi](\tau)$ by the following lemma.
\begin{lem}
 \label{lgbarE}
\begin{equation*}
E[\phi](\tau)\les \bar g(0, \tau)+2\int_{r\leq R}|\pa\phi|^2+\phi^2
dx\les \tilde{E}[\phi](\tau).
\end{equation*}
\end{lem}
\begin{proof}
In fact note that
\begin{equation*}
\begin{split}
 \bar g(0, \tau)+2\int_{r\leq R}|\pa\phi|^2+\phi^2 dx&=\int_{S_\tau}r^2(\pa_v\phi)^2+\pa_v(r\phi^2) +r^2|\nabb\phi|^2dvd\om +2\int_{r\leq R}|\pa\phi|^2+\phi^2 dx\\
&=E[\phi](\tau)+\left.\int_{\om}r\phi^2d\om\right|_{v_\tau}^{\infty}+\int_{r\leq
R}|\pa\phi|^2+2\phi^2dx.
\end{split}
\end{equation*}
Lemma ~\ref{sphereinbd} and Lemma \ref{phiboundH} imply that
\[
 \bar g(0, \tau)+2\int_{r\leq R}|\pa\phi|^2+\phi^2 dx\les
 \tilde{E}[\phi](\tau).
\]
To prove the other side of the inequality, it suffices to show that
\[
\left.\int_{\om}r\phi^2(\tau, R, \om)d\om\right|_{r=R}\leq
\int_{r\leq R}|\pa_r\phi|^2+2\phi^2 dx.
\]
Without loss of generality, assume $R\geq 2$. Notice that
\begin{equation*}
 R^3\int_{\om}\phi^2(\tau, R, \om)d\om=\int_{0}^{R}\int_{\om}\pa_r(r^3\phi^2)d\om dr\leq 3\int_{r\leq R}
 \phi^2dx + R\int_{r\leq R}|\pa_r\phi|^2+\phi^2dx.
\end{equation*}
 Hence
\[
 R\int_{\om}\phi^2(\tau, R, \om)d\om\leq \int_{r\leq R}|\pa_r\phi|^2+2\phi^2 dx\leq \int_{r\leq R}|\pa\phi|^2+2\phi^2
 dx.
\]
Thus the lemma holds.
\end{proof}
Since $G[0, 0]_{\tau_1}^{\tau_2}\leq \bar G[0, 0]_{\tau_1}^{\tau_2}$, using Lemma \ref{lgbarE}, we can control $\bar{G}[0, 0]$ in terms of $E[\phi](\tau)$
\[
 G[0, 0]_{\tau_1}^{\tau_2}\leq\bar G[0, 0]_{\tau_1}^{\tau_2}\leq \int_{\tau_1}^{\tau_2}\tilde{E}[\phi](\tau)d\tau.
\]
Summarizing, we can show that
\begin{equation}
 \begin{split}
  &\int_{\tau_1}^{\tau_2}\int_{S_\tau}r^p|\pa_v \Phi||\pa\psi||\pa_v\psi|+r^{p+1}|L(\pa\phi)||\pa_v\psi|dvd\om d\tau\\
&\les
\left(\ep_2+\frac{\tau_1^{-\a}}{\ep_2}+\frac{\tau_1^{-2\a}}{\ep_2^2\ep_3}\right)\int_{\tau_1}^{\tau_2}
E[\phi](\tau)d\tau+\tau_1^{1-\a}E[\phi](\tau_1)+\int_{\tau_1}^{\tau_2}\tau^{-\a}D^\a[F]_{\tau}^{\tau_2}d\tau\\
&\qquad+\tau_1^{1-\a}
D^\a[F]_{\tau_1}^{\tau_2}+\frac{1}{\ep_2}G[1+2\a,
1+2\a]_{\tau_1}^{\tau_2}+ \ep_3\bar G[2\a,
2\a]_{\tau_1}^{\tau_2}+G[1+\a, p]_{\tau_1}^{\tau_2}.
 \end{split}
\label{ppavPhi}
\end{equation}
Here we used the argument in the end of previous section to replace $\tilde{E}[\phi](\tau)$ with $E[\phi](\tau)$. We must
remark here that the implicit constants before the other terms on
the right hand side of \eqref{ppavPhi} may also depend on $\ep_i$.
However, since $\ep_i$ will be chosen to depend only on $R$, $\a$,
$B^{\a\b}$, $C_0$, the omitted dependence will not affect the argument in the sequel.

\bigskip

Finally, we treat the inhomogeneous term $r^{p+1}F\cdot \pa_v
\psi$ in \eqref{pWE}. Since $D^{2\a}[F]_{\tau_1}^{\tau_2}$ is expected to decay in $\tau$, we put some positive weights of $\tau$
in $D^{2\a}[F]_{\tau}^{\tau_2}$ and estimate it by using Lemma
\ref{lweightILE} applied to $\b=p-\frac{3}{2}\a$,
$f(\tau)=\int_{S_\tau}r^{1+2\a}|F|^2d\si$. We can show that
\begin{equation}
\label{pWEFi}
\begin{split}
 &\int_{\tau_1}^{\tau_2}\int_{S_\tau}2r^{p+1}F\cdot\pa_v\psi dvd\tau d\om\\
 &\leq
 \ep_4 G[p-\frac{3}{2}\a, 2p-1-2\a]_{\tau_1}^{\tau_2} +
 \frac{1}{\ep_4}\int_{\tau_1}^{\tau_2}(1+\tau)^{p-\frac{3}{2} \a}\int_{S_\tau}|F|^2
 r^{1+2\a}d\vol\\
 &\les \ep_4 G[p-\frac{3}{2}\a,
 2p-1-2\a]_{\tau_1}^{\tau_2}+\frac{1}{\ep_4}\int_{\tau_1}^{\tau_2}\tau^{p-1-\frac{3}{2}\a}D^{2\a}[F]_{\tau}^{\tau_2}d\tau
 +\frac{1}{\ep_4}\tau_1^{p-\frac{3}{2}\a}D^{2\a}[F]_{\tau_1}^{\tau_2}
 \end{split}
\end{equation}
for any $0<\ep_4\leq 1$ and $\tau_2\geq\tau_1\geq t_0$.

\subsection{Proof of Proposition \ref{mainprop}}
Having controlled $\int_{\tau_1}^{\tau_2}\int_{S_\tau}r^{p+1}(F-N-L)\pa_v\psi dvd\om d\tau$, we are now
able to prove Proposition \ref{mainprop}. First let
$$\ep_1= \frac{1-2\a}{2A},\quad A=\max \{|B^{\mu\nu}|\}.$$
Hence for $p=1$ or $1+2\a$, the third term in ~\eqref{pWE} dominates
the first term on the right hand side of ~\eqref{Phibarpsi}. Set
$p=1+2\a$ in ~\eqref{pWE} and $\ep_4=1$ in ~\eqref{pWEFi}. Combining
the estimates ~\eqref{bdestR}, ~\eqref{Phibarpsi}, ~\eqref{Phiphip},
~\eqref{ppavPhi}, we infer that
\begin{equation*}
 \begin{split}
  g(1+2\a, \tau_2)+ \bar G[0, 2\a]_{\tau_1}^{\tau_2}&\les g(1+2\a, \tau_1)+\left(\ep_2+\frac{\tau_1^{-\a}}{\ep_2}+\frac{\tau_1^{-2\a}}{\ep_2^2\ep_3}\right)\int_{\tau_1}^{\tau_2}
E[\phi](\tau)d\tau+\int_{\tau_1}^{\tau_2}\tau^{\f12\a}D^{2\a}[F]_{\tau}^{\tau_2}d\tau\\
&+\tau_1^{1-\a}E[\phi](\tau_1)+\tau_1^{1+\f12\a}
D^{2\a}[F]_{\tau_1}^{\tau_2}+ \ep_3\bar G[0, 2\a]_{\tau_1}^{\tau_2}
+\frac{1}{\ep_2}G[1+\f12\a, 1+2\a]_{\tau_1}^{\tau_2}.
 \end{split}
\end{equation*}
 Now suppose the implicit constant before
$\ep_3 \bar G[0, 2\a]$ is $C_1$, which is independent of $\ep_2$,
$\ep_3$. Take
$$\ep_3=\frac{1}{2C_1}.$$
We remark here that we can choose different $\ep_i$ for different
values of $p$. In particular, we conclude that $\ep_3 \bar G[0, 2\a]$ can be absorbed by the
left hand side. Then apply Gronwall's inequality( Lemma \ref{lGronwall}). We can control the
last term $\frac{1}{\ep_2}G[1+\f12\a,
1+2\a]_{\tau_1}^{\tau_2}$ and conclude that
\begin{equation}
\label{pWE1abarG}
 \begin{split}
  g(1+2\a, \tau_2)+ \bar G[0, 2\a]_{\tau_1}^{\tau_2}&\les g(1+2\a, \tau_1) +\tau_1^{1+\f12\a} D^{2\a}[F]_{\tau_1}^{\tau_2}+\int_{\tau_1}^{\tau_2}\tau^{\f12\a}D^{2\a}[F]_{\tau}^{\tau_2}d\tau\\
&+\exp\left(\frac{2\tau_1^{-\f12\a}}{\a\ep_2}\right)\left(\ep_2+\frac{\tau_1^{-\a}}{\ep_2}+\frac{\tau_1^{-2\a}}{\ep_2^2}\right)\int_{\tau_1}^{\tau_2}
E[\phi](\tau)d\tau+\tau_1^{1-\a}E[\phi](\tau_1).
 \end{split}
\end{equation}
The integral of the energy on the right hand side can be estimated
when we combine \eqref{pWE1abarG} with the p-weighted energy
inequality for $p=1$.

\bigskip

Now take $p=1$ in \eqref{pWE}. First, we use interpolation to estimate
the first term $G[1-\frac{3}{2}\a, 1-2\a]_{\tau_1}^{\tau_2}$ on the
right hand side of ~\eqref{pWEFi}
\[
 G[1-\frac{3}{2}\a, 1-2\a]_{\tau_1}^{\tau_2}\leq\left(G[1+\f12 \a, 1]_{\tau_1}^{\tau_2}\right)^{1-2\a}
\left(G[0, 0]_{\tau_1}^{\tau_2}\right)^{2\a}\leq G[1+\f12 \a,
1]_{\tau_1}^{\tau_2}+G[0, 0]_{\tau_1}^{\tau_2}.
\]
To retrieve the full energy $E[\phi](\tau)$ from $\bar{g}(0, \tau)$,
by Lemma \ref{lgbarE}, add
\[
 2\int_{\tau_1}^{\tau_2}\int_{r\leq R} |\pa\phi|^2+\phi^2dxd\tau
\]
to both sides of ~\eqref{pWE}. Then the integrated local energy
estimate ~\eqref{ILEG} restricted to the region $r\leq R$ and
Gronwall's inequality imply that
\begin{align*}
  g(1, \tau_2)+ \int_{\tau_1}^{\tau_2}E[\phi](\tau)d\tau&\les g(1, \tau_1)+\left(\ep_2+\frac{\tau_1^{-\a}}{\ep_2}+\ep_4+\frac{\tau_1^{-2\a}}{\ep_2^2}\right)\int_{\tau_1}^{\tau_2}
E[\phi](\tau)d\tau+\int_{\tau_1}^{\tau_2}\tau^{-\a}D^{2\a}[F]_{\tau}^{\tau_2}d\tau\\
&+\tau_1^{1-\a}E[\phi](\tau_1)+\tau_1^{1-\a}
D^{2\a}[F]_{\tau_1}^{\tau_2}+\tau_1^{-2\a} \bar G[0,
2\a]_{\tau_1}^{\tau_2}+\frac{1}{\ep_2}G[1+2\a,
1+2\a]_{\tau_1}^{\tau_2},
\end{align*}
where we choose $\ep_3=1$. Assume the implicit constant before
$\int_{\tau_1}^{\tau_2}E[\phi](\tau)d\tau$ in the above inequality
is $C_3$, which is independent of $\ep_2$ and $\ep_4$. Then take
\[\ep_4=\frac{1}{2C_3}.\]
We get
\begin{align}
\notag
  &g(1, \tau_2)+ \int_{\tau_1}^{\tau_2}E[\phi](\tau)d\tau\les g(1, \tau_1)+\left(\ep_2+\frac{\tau_1^{-\a}}{\ep_2}+\frac{\tau_1^{-2\a}}{\ep_2^2}\right)\int_{\tau_1}^{\tau_2}
E[\phi](\tau)d\tau+\int_{\tau_1}^{\tau_2}\tau^{-\a}D^{2\a}[F]_{\tau}^{\tau_2}d\tau\\
\label{pWE1barG}
&\qquad\qquad\qquad+\tau_1^{1-\a}E[\phi](\tau_1)+\tau_1^{1-\a}
D^{2\a}[F]_{\tau_1}^{\tau_2}+\tau_1^{-2\a} \bar G[0,
2\a]_{\tau_1}^{\tau_2}+\frac{1}{\ep_2}G[1+2\a,
1+2\a]_{\tau_1}^{\tau_2}.
\end{align}
Now let $C_4$ be the implicit constant before
$\int_{\tau_1}^{\tau_2} E[\phi](\tau)d\tau$ in both
~\eqref{pWE1abarG} and ~\eqref{pWE1barG}, which is independent of
$\ep_2$. Then let
\[
 \ep_2=\frac{1}{4C_4}
\]
and choose a constant $T_0\geq t_0$ such that
\[
 T_0^{-\f12\a}\leq\frac{\a}{2}\ep_2.
\]
In particular, for $\tau_1\geq T_0$, we have
\[
 C_4\left(\ep_2+\frac{\tau_1^{-\a}}{\ep_2}+\frac{\tau_1^{-2\a}}{\ep_2^2}\right)
 \leq C_4\left(\ep_2+\frac{T_0^{-\a}}{\ep_2}+\frac{T_0^{-2\a}}{\ep_2^2}\right)\leq
 \f12.
\]
We combine ~\eqref{pWE1abarG} and ~\eqref{pWE1barG} together to control $\int_{\tau_1}^{\tau_2}E[\phi](\tau)d\tau$.
For $\tau_2\geq \tau_1\geq T_0$, we first estimate $\bar G[0,
2\a]_{\tau_1}^{\tau_2}$, $G[1+2\a, 1+2\a]_{\tau_1}^{\tau_2}$ in
~\eqref{pWE1barG} by using ~\eqref{pWE1abarG}. Then combining all them
together, we can show that the coefficient of
$\int_{\tau_1}^{\tau_2}E[\phi](\tau)d\tau$ on the right hand side
can be bounded by
\[
\f12+C_4T_0^{-2\a}
\frac{e}{2}+\frac{C_4}{\ep_2}\frac{T_0^{-2\a}}{2\a}\frac{e}{2}<\frac{3}{4}.
\]
Thus $\int_{\tau_1}^{\tau_2}E[\phi](\tau)d\tau$ can be absorbed and we can conclude that
\begin{align*}
  g(1, \tau_2)+ \int_{\tau_1}^{\tau_2}E[\phi](\tau)d\tau&\les g(1, \tau_1)+\tau_1^{1-\a} D^{2\a}[F]_{\tau_1}^{\tau_2}+\tau_1^{1-\a}E[\phi](\tau_1)+\int_{\tau_1}^{\tau_2}\tau^{-\a}D^{2\a}[F]_{\tau}^{\tau_2}d\tau\\
&\qquad+\tau_1^{-2\a}\int_{\tau_1}^{\tau_2}\tau^{\f12\a}D^{2\a}[F]_{\tau}^{\tau_2}d\tau+\tau_1^{-2\a}g(1+2\a,
\tau_1),
\end{align*}
which, in turn, improves ~\eqref{pWE1abarG} to
\begin{equation*}
 \begin{split}
  g(1+2\a, \tau_2)+ \bar G[0, 2\a]_{\tau_1}^{\tau_2}&\les g(1+2\a, \tau_1)+\tau_1^{1-\a}E[\phi](\tau_1)+
  \tau_1^{1+\f12\a}
  D^{2\a}[F]_{\tau_1}^{\tau_2}+\int_{\tau_1}^{\tau_2}\tau^{\f12\a}D^{2\a}[F]_{\tau}^{\tau_2}d\tau.
 \end{split}
\end{equation*}
This proves ~\eqref{pWE1} and ~\eqref{pWE1a} for
all $\tau_2\geq\tau_1\geq T_0$.

\bigskip

For $t_0\leq \tau_1\leq \tau_2\leq T_0$, we make use of the boundedness of $\tau$. Let $\ep_2=1$. Inequality
~\eqref{pWE1abarG} shows that
\begin{equation*}
 g(1+2\a, \tau_2)+\bar G[0, 2\a]_{\tau_1}^{\tau_2}\les g(1+2\a, \tau_1)+\int_{\tau_1}^{\tau_2} E[\phi]
 (\tau)d\tau+ D^{2\a}[F]_{\tau_1}^{\tau_2}+E[\phi](\tau_1).
\end{equation*}
By Lemma ~\ref{lgbarE}, we have
\begin{equation*}
 \bar G[1+2\a, \a]_{\tau_1}^{\tau_2}\les \bar G[0, 2\a]_{\tau_1}^{\tau_2}+ \bar G[0,0]_{\tau_1}^{\tau_2}
 \les g(1+2\a, \tau_1)+\int_{\tau_1}^{\tau_2} E[\phi](\tau)d\tau+
 D^{2\a}[F]_{\tau_1}^{\tau_2}+E[\phi](\tau_1).
\end{equation*}
Combining with the energy inequality ~\eqref{EING}, we obtain
\[
 E[\phi](\tau_2)\les g(1+2\a, \tau_1)+\int_{\tau_1}^{\tau_2} E[\phi](\tau)d\tau+
 D^{2\a}[F]_{\tau_1}^{\tau_2}+E[\phi](\tau_1).
\]
Thus Gronwall's inequality indicates that
\begin{equation*}
 \int_{\tau_1}^{\tau_2}E[\phi](\tau)d\tau\les g(1+2\a, \tau_1)+
 D^{2\a}[F]_{\tau_1}^{\tau_2}+E[\phi](\tau_1)
\end{equation*}
as $\tau_1\leq \tau_2\leq T_0$. Hence \eqref{pWE1} and
\eqref{pWE1a}
 follow from ~\eqref{pWE1abarG}, ~\eqref{pWE1barG}.

 \bigskip

 For $\tau_1\leq \tau_2\leq t_0$, the finite speed of propagation for
wave equation ~\cite{sogge} shows that $g(p, \tau)$ vanishes. Thus
~\eqref{pWE1}, ~\eqref{pWE1a} hold. For general $\tau_2\geq
\tau_1\geq 0$, divide the interval $[\tau_1, \tau_2]$ into
three(possibly two) such intervals: $[\tau_1, t_0]$, $[t_0, T_0]$
and $[T_0, \tau_2]$. Then ~\eqref{pWE1}, ~\eqref{pWE1a} follow by
combining those three(or two) inequalities together. This completes the proof for ~\eqref{pWE1}, ~\eqref{pWE1a}.

\bigskip

Having proven ~\eqref{pWE1} and ~\eqref{pWE1a}, we can improve the
integrated local energy inequality ~\eqref{EING} and the energy
inequality ~\eqref{ILEG} as follows: Integrate \eqref{pWE1a} from
$\tau_1$ to $\tau_2$. We obtain
\begin{align*}
 &G[2+2\a, 1+2\a]_{\tau_1}^{\tau_2}+\bar G[1+2\a, \a]_{\tau_1}^{\tau_2}\leq G[2+2\a, 1+2\a]_{\tau_1}^{\tau_2}+(1+\tau_1)^{-1-2\a} \bar G[0, 2\a]_{\tau_1}^{\tau_2}\\
&\les
E[\phi](\tau_1)+\tau_1^{-\a}D^{2\a}[F]_{\tau_1}^{\tau_2}+(1+\tau_1)^{-1-2\a}
\left(g(1+2\a,
\tau_2)+\int_{\tau_1}^{\tau_2}\tau^{\f12\a}D^{2\a}[F]_{\tau}^{\tau_2}d\tau\right),
\end{align*}
which, together with \eqref{EING}, ~\eqref{ILEG}, implies ~\eqref{ILE0}, ~\eqref{eb}. We thus finished proof for Proposition
\ref{mainprop}.

\bigskip

To show Corollary ~\ref{D2aQp}, take $p=1+2\a$ in ~\eqref{phi21a}.
Interpolation shows that
\begin{equation*}
\int_{S_\tau}r^{2\a}\phi^2dvd\om\les
E[\phi](\tau)+E[\phi](\tau)^{1-\frac{2\a}{1-\a}}g(1+2\a,
\tau)^{\frac{2\a}{1-\a}}.
\end{equation*}
Using Jensen's inequality, we have
\begin{align*}
 \int_{\tau_1}^{\tau_2}\int_{S_\tau}\frac{r^{2\a}\phi^2}{(1+\tau)^{1+8\a}}dvd\om d\tau&
 \les \int_{\tau_1}^{\tau_2}\frac{E[\phi](\tau)}{(1+\tau)^{1+2\a}}d\tau+G[4-\a,
 1+2\a]_{\tau_1}^{\tau_2}.
\end{align*}
Therefore for $\tau_2\geq \tau_1\geq t_0$, Lemma ~\ref{nullQ} and
Proposition ~\ref{mainprop} imply that
\begin{align*}
 D^{2\a}[N]_{\tau_1}^{\tau_2}&=\int_{\tau_1}^{\tau_2}\int_{\Si_\tau}|B^{\a\b}\pa_\a\Phi\cdot\pa_\b\phi|^2(1+r)^{1+2\a}d\vol\\
&\les\int_{\tau_1}^{\tau_2}\int_{\Si_\tau}\frac{|\bar\pa\phi|^2}{(1+r)^{1+\a}}d\vol+\int_{\tau_1}^{\tau_2}\int_{S_\tau}|rN|^2 r^{1+2\a}dvd\om d\tau\\
&\les \int_{\tau_1}^{\tau_2}\int_{\Si_\tau}\frac{|\bar\pa\phi|^2}{(1+r)^{1+\a}}d\vol+\int_{\tau_1}^{\tau_2}\int_{S_\tau}\frac{r^{2\a}|\overline{\pa_v}\psi|^2}{(1+\tau)^{1+8\a}}+\frac{|\pa\psi|^2}{(1+r)^{1+4\a}}+\frac{r^{2\a}\phi^2}{(1+\tau)^{1+8\a}} dvd\om d\tau\\
&\les \int_{\tau_1}^{\tau_2}\int_{\Si_\tau}\frac{|\bar\pa\phi|^2}{(1+r)^{1+\a}}d\vol+\bar G[1+8\a, 2\a]_{\tau_1}^{\tau_2}+\int_{\tau_1}^{\tau_2}\frac{E[\phi](\tau)}{(1+\tau)^{1+2\a}}d\tau+G[2+\a, 1+2\a]_{\tau_1}^{\tau_2}\\
&\les E[\phi](\tau_1)+D^{\a}[F]_{\tau_1}^{\tau_2}+(1+\tau_1)^{-\a}D^{2\a}[F]_{\tau_1}^{\tau_2}\\
&\qquad+(1+\tau_1)^{-1-2\a}\left(g(1+2\a,
\tau_2)+\int_{\tau_1}^{\tau_2}\tau^{\f12\a}
D^{2\a}[F]_{\tau}^{\tau_2}d\tau\right).
\end{align*}
For $\tau_1\leq \tau_2\leq t_0$, notice that
\[
 D^{2\a}[N]_{\tau_1}^{\tau_2}\les\int_{\tau_1}^{\tau_2}E[\phi](\tau)d\tau.
\]
For the linear terms $L(\pa\phi)$, we can show
\begin{align*}
D^{2\a}[L]_{\tau_1}^{\tau_2}=\int_{\tau_1}^{\tau_2}\int_{\Si_\tau}|L(\pa\phi)|^2(1+r)^{1+2\a}d\vol&\les
\int_{\tau_1}^{\tau_2}\int_{\Si_\tau}
(1+r)^{-2-6\a}|\pa\phi|^2(1+r)^{1+2\a}d\vol\\
&\les
\int_{\tau_1}^{\tau_2}\int_{\Si_\tau}\frac{|\pa\phi|^2}{(1+r)^{1+\a}}d\vol.
\end{align*}
The corollary then follows from \eqref{ILE0} and ~\eqref{eb}.

\section{Decay of the Solution}
Under appropriate assumptions on the inhomogeneous term $F$,
Proposition \ref{mainprop} leads to the decay of the energy flux
$E[\phi](\tau)$. After commuting the equation with the vector fields
$Z$, we obtain the pointwise decay of the solution outside the
cylinder $\{(t, x)| |x|\leq R\}$ by using Sobolev embedding and
inside the cylinder by using elliptic estimates.

\begin{prop}
 \label{energydecay}
Suppose there is a constant $C_1$ such that
\[
 D^{2\a}[F]_{\tau_1}^{\tau_2}\leq C_1(1+\tau_1)^{-1-\a}, \qquad \forall \tau_2\geq \tau_1\geq
 0.
\]
Then for solution $\phi$ of the linear wave equation
\eqref{LWAVEEQ}, we have energy flux decay
\[
 E[\phi](\tau)\les (\ep^2 E_0+C_1)(1+\tau)^{-1-\a}.
\]
\end{prop}
\begin{proof} Since the initial data are supported in the region $\{|x|\leq R_0\leq R\}$, the finite speed of propagation shows that $g(1+2\a, 0)$ vanishes. Take $\tau_1=0$ in ~\eqref{pWE1a}. We get
\begin{equation}
\label{pwe1a}
 g(1+2\a, \tau)=\int_{S_\tau}r^{1+2\a}(\pa_v\psi)^2dvd\om\les C_1+\ep^2E_0
\end{equation}
and
\begin{equation}
\label{pwe1ai}
 \int_{\tau_1}^{\tau_2}\int_{S_\tau}r^{2\a} (\pa_v\psi)^2dvd\om d\tau\leq\bar G[0, 2\a]_{0}^{\tau_2}\les C_1+\ep^2
 E_0.
\end{equation}
We claim that we can choose a dyadic sequence
$\{\tau_n\rightarrow\infty\}$ such that
\begin{equation}
\label{dyadicseq}
 \int_{S_{\tau_n}}r^{2\a}(\pa_v\psi)^2dvd\om\leq(1+\tau_n)^{-1}\left(C_1+\ep^2
 E_0\right),
\end{equation}
where $\tau_n$ satisfies the inequality
$\ga^{-2}\tau_n\leq\tau_{n-1}\leq\ga^2\tau_n$ for some large
constant $\ga$. In fact, there exists $\tau_n\in[\ga^{n},
\ga^{n+1}]$ such that \eqref{dyadicseq} holds. Otherwise
$$\int_{\ga^{n}}^{\ga^{n+1}}\int_{S_\tau}r^{2\a}(\pa_v\psi)^2dvd\om d\tau \geq \ln\ga \left(\ep^2 E_0+ C_1\right),$$
which contradicts to ~\eqref{pwe1ai} if $\ga$ is large enough.

Take $\tau=\tau_n$ in ~\eqref{pwe1a}. Interpolate with ~\eqref{dyadicseq}. We obtain
\begin{equation*}
 \int_{S_{\tau_n}}r(\pa_v\psi)^2dvd\om\les (1+\tau_n)^{-2\a}\left(\ep^2 E_0
 +C_1\right).
\end{equation*}
Then the inequality ~\eqref{pWE1} implies that for $\tau\geq\tau_n$
\begin{equation}
\label{pwe10}
 \begin{split}
  \int_{S_\tau}r(\pa_v\psi)^2dvd\om+\int^{\tau}_{\tau_{n}}E[\phi](s)ds&\les
   (1+\tau_{n})^{-2\a}\left(\ep^2E_0+C_1\right)+\tau_{n}^{1-\a}E[\phi](\tau_{n}).
 \end{split}
\end{equation}
On the other hand the energy inequality \eqref{eb} shows that for
all $s\leq \tau$
\[
 E[\phi](\tau)\les E[\phi](s)+(1+s)^{-1-\a}\left(\ep^2 E_0
 +C_1\right).
\]
In particular
\[
 E[\phi](\tau_1)\les E[\phi](0)+\ep^2 E_0 +C_1\les \ep^2 E_0 +C_1.
\]
By ~\eqref{pwe10}, we have
\begin{equation}
\label{tautaun}
 (\tau-\tau_n)E[\phi](\tau)-\int^{\tau}_{\tau_{n}}(1+s)^{-1-\a}\left(\ep^2 E_0 +C_1\right)ds
 \les(1+ \tau_n)^{-2\a}\left(\ep^2
 E_0+C_1\right)+\tau_n^{1-\a}E[\phi](\tau_n).
\end{equation}
In particular for $n=1$
\begin{equation*}
 E[\phi](\tau)\les (1+\tau)^{-1}\left(\ep^2 E_0+C_1\right).
\end{equation*}
Let $\tau=\tau_{n+1}$ in ~\eqref{tautaun}. We obtain
\begin{equation*}
 (\tau_{n+1}-\tau_n)E[\phi](\tau_{n+1})\les (1+\tau_{n})^{-\a}\left(\ep^2
 E_0+C_1\right).
\end{equation*}
Since $\tau_n$ are dyadic, we have
\begin{equation*}
 E[\phi](\tau_{n})\les \tau_{n}^{-1-\a}\left(\ep^2 E_0+C_1\right),\quad \forall
 n.
\end{equation*}
Finally, for $\tau\in[\tau_n, \tau_{n+1}]$, we can show
\[
 E[\phi](\tau)\les E[\phi](\tau_n)+(1+\tau_n)^{-1-\a}\left(\ep^2 E_0+C_1\right)\les (1+\tau_n)^{-1-\a}
 \left(\ep^2 E_0+C_1\right)\les (1+\tau)^{-1-\a}\left(\ep^2
 E_0+C_1\right).
\]
\end{proof}
With the energy flux decay, we can obtain the decay of the spherical
average of the solution.
\begin{cor}
 \label{ptdcoutc}
Assume that there is a constant $C_1$ such that
\[
 D^{2\a}[F]_{\tau_1}^{\tau_2}\leq C_1(1+\tau_1)^{-1-\a}, \qquad \forall \tau_2\geq \tau_1\geq
 0.
\]
Then on the hypersurface $S_\tau$, we have
\begin{align*}
&\int_{\om}|r\phi|^2d\om\les\ep^2 E_0 + C_1, \qquad\qquad \qquad
\qquad r\geq R, \\
&\int_{\om}r|\phi|^2d\om\les(1+\tau)^{-1-\a}\left(\ep^2 E_0 +
C_1\right), \qquad r\geq R.
\end{align*}
\end{cor}
\begin{proof} By Proposition \ref{energydecay}, the first inequality follows from ~\eqref{phi2bd} and
~\eqref{pwe1a}. The second one follows from Lemma ~\ref{lem1}.
\end{proof}

In order to obtain the pointwise decay of the solution which is
usually a consequence of Sobolev embedding, we need energy estimates for
the derivative of the solution. For this purpose, we commute the equation with
the vector fields $\Om$ and $T$. Under appropriate assumptions on
the inhomogeneous term $F$, we hope to derive the same energy decay
for $\Om^k T^j\phi$. Denote
\[
 N(\phi_1, \phi_2)=B^{\mu\nu}\pa_\mu\phi_1\cdot\pa_\nu\phi_2, \qquad \forall \phi_1, \phi_2 \in
 C^{\infty}(\mathbb{R}^{3+1}),
\]
where we recall that the constants $B^{\a\b}$ satisfy the null
condition.

\begin{lem}
\label{nullstructure} Let $Z$ be $\Om$ or $T$. Then
\[
 Z^\b N(\phi_1, \phi_2)=\sum\limits_{\b_1+\b_2=\b} N(Z^{\b_1}\phi_1,
 Z^{\b_2}\phi_2).
\]
\end{lem}
\begin{proof}
 Notice that $[\Om, \pa_r]=[\Om, \pa_t]=[\Om, \nabb]=0$. The lemma then follows from the fact
 that $B^{\a\b}$ satisfy the null condition.
\end{proof}

Based on Corollary ~\ref{D2aQp}, we are able to prove the decay of
the energy flux of $Z^\b\phi$ after commuting the linear equation
\eqref{LWAVEEQ} with $Z^\b$.
\begin{prop}
\label{energydecaycom}
 Assume that there is a constant $C_1$ such that the inhomogeneous term $F$ in ~\eqref{LWAVEEQ} satisfies
  the following condition
\[
D^{2\a}[Z^\b F]_{\tau_1}^{\tau_2}\leq C_1(1+\tau_1)^{-1-\a}, \quad
\forall \tau_2\geq \tau_1\geq 0, \quad \forall \b\leq \b_0
\]
for some multiple indices $|\b_0|\leq 4$. Assume $\Phi$, $L^\mu(t,
x)$ satisfy the conditions in Theorem \ref{maintheorem2}. Then we
have
\begin{align}
 \label{inductionE}
&E[Z^\b\phi](\tau)\les \left(C_1+\ep^2 E_0\right)(1+\tau_1)^{-1-\a},\\
\label{inductionD} &D^{2\a}[N(Z^{\b_1}\Phi,
Z^{\b}\phi)]_{\tau_1}^{\tau_2}+D^{2\a}[Z^{\b_1}L^\mu\cdot
Z^{\b}\pa_\mu\phi)]_{\tau_1}^{\tau_2}\les \left(C_1+\ep^2
E_0\right)(1+\tau_1)^{-1-\a}
\end{align}
for $\forall \b\leq \b_0$, $|\b_1|\leq 4$.
\end{prop}
\begin{proof} We prove the proposition by induction. When $\b=0$, ~\eqref{inductionE} follows from Proposition
~\ref{energydecay}. Since $Z^{\b_1}\Phi$ is $(\delta, \a, t_0, R_1,
C_0)$-weak wave, $\forall |\b_1|\leq 4$,  Corollary ~\ref{D2aQp} and
inequality ~\eqref{pwe1a} imply that
\[
 D^{2\a}[N(Z^{\b_1}\Phi, \phi)]_{\tau_1}^{\tau_2}+D^{2\a}[Z^{\b_1}L^\mu\cdot\pa_\mu\phi]_{\tau_1}^{\tau_2}
 \les (C_1+\ep^2 E_0)(1+\tau_1)^{-1-\a},\quad \forall |\b_1|\leq 4.
\]
Assume that ~\eqref{inductionE}, ~\eqref{inductionD} hold for all
$\b'<\b$. Commute the equation \eqref{LWAVEEQ} with $Z^{\b}$. Using
Lemma ~\ref{nullstructure}, we have the equation for $Z^\b\phi$
\begin{equation}
\label{waveeqomktj}
 \Box(Z^\b\phi)+N(\Phi,Z^\b \phi)+L(Z^\b\phi)=Z^\b F-\sum\limits_{\b_1+\b_2\leq\b, \b_2<\b}
 N(Z^{\b_1}\Phi, Z^{\b_2}\phi)+Z^{\b_1}L^\mu\cdot Z^{\b_2}\pa_\mu\phi.
\end{equation}
Since $\b_2<\b$, by the induction assumptions, we get
\begin{equation*}
 D^{2\a}\left[Z^\b F-\sum\limits_{\b_2<\b} N(Z^{\b_1}\Phi, Z^{\b_2}\phi)+Z^{\b_1}L^\mu\cdot
  Z^{\b_2}\pa_\mu\phi\right]_{\tau_1}^{\tau_2}\les\left(C_1+\ep^2 E_0\right)(1+\tau_1)^{-1-\a}.
\end{equation*}
Hence for $Z^\b\phi$, inequality ~\eqref{inductionE} follows
from Proposition ~\ref{energydecay} and inequality \eqref{inductionD} follows from Corollary ~\ref{D2aQp} and
Proposition ~\ref{mainprop}.
\end{proof}

\bigskip

Since the angular momentum $\Om$ is vanishing for $r=0$, we are not able to obtain the pointwise bound of the solution
in the cylinder $\{|x|\leq R\}$ by commuting the equation with $\Om$. We instead rely on elliptic estimates and
 the vector $T=\pa_t$ as commutators.
\begin{lem}
 \label{H2phil}
Assume that there is a constant $C_1$ such that
\[
 D^{2\a}[F]_{\tau_1}^{\tau_2}+ D^{2\a}[\pa_t F]_{\tau_1}^{\tau_2}\leq C_1(1+\tau_1)^{-1-\a},
 \qquad \forall \tau_2\geq \tau_1\geq 0.
\]
Then for solution of the linear wave equation ~\eqref{LWAVEEQ}, we
have
\begin{equation*}
\int_{r\leq R}|\pa^2\phi|^2dx=\sum\limits_{\mu,
\nu=0}^{3}\int_{r\leq R}|\pa_{\mu\nu}\phi|^2dx\les
\left(E_0\ep^2+C_1\right)(1+\tau)^{-1-\a}.
\end{equation*}
\end{lem}
\begin{proof} We first assume that $\tau\geq R$. Take $\b_0=(0, 1)$ in Proposition ~\ref{energydecaycom}. We have
\[
 E[T^j\phi](\tau_1)+\int_{\tau_1}^{\tau_2}\int_{\Si_\tau}\frac{|\bar\pa T^j\phi|^2}{(1+r)^{1+\a}}dxd\tau+
 D^{2\a}[T^j N(\Phi, \phi)+T^j L(\pa\phi)]_{\tau_1}^{\tau_2}\les
 \left(E_0\ep^2+C_1\right)(1+\tau_1)^{-1-\a}
\]
for all $j\leq 1$. Using elliptic estimates ~\cite{elliptic}, we
can show that
\begin{align}
\notag
 &\int_{r\leq R}|\pa^2\phi|^2dx=\int_{r\leq R}\sum\limits_{i, j=1}^{3}|\pa_{ij}\phi|^2dx+2\sum\limits_{\a=0}^{3}\int_{r\leq R}|\pa_\a\pa_t\phi|^2dx\\
\notag
&\qquad\les E[\pa_t\phi](\tau)+\int_{r\leq 2R}|\Delta \phi|^2+|\phi|^2dx\\
\label{H2phi2R}
&\qquad\les E[\pa_t\phi](\tau)+\int_{r\leq 2R}|\pa_{tt}\phi+F-N(\Phi, \phi)-L(\pa\phi)|^2+\phi^2dx\\
\notag &\qquad\les E[\pa_t\phi](\tau)+\int_{r\leq
2R}|\pa_{tt}\phi|^2dx+ \sum\limits_{j\leq
1}\int_{\tau}^{\tau+R}\int_{r\leq 2R}|T^j
F|^2+|T^j\phi|^2+|T^jN|^2+|T^j L|^2dxdt.
\end{align}
Consider the region bounded by $\Si_{\tau-R}$ and $t=\tau$. Take
$X=T$ in ~\eqref{energyeq}. Note that the vector field $T$ is
killing, that is $K^T[\pa_t\phi]=0$. We can conclude that
\begin{equation*}
\label{energyineq2R}
\begin{split}
\int_{r\leq 2R}J^T_\mu[\pa_t\phi]n^\mu d\si&=\int_{\Si_{\tau-R}\cap
\{t\leq \tau\}}J^T_\mu[\pa_t\phi]n^\mu d\si
+\int_{\tau-R}^{\tau}\int_{r\leq R +t-\tau}(TN+TL-\pa_t
F)\pa_{tt}\phi d\vol.
\end{split}
\end{equation*}
Apply Cauchy-Schwartz inequality to the last term. We obtain
\begin{equation*}
 \begin{split}
 & \qquad\int_{\tau-R}^{\tau}\int_{r\leq R +t-\tau}|\pa_t F-TN-TL||\pa_{tt}\phi|d\vol\\
&\les\int_{\tau-R}^{\tau}\int_{\Si_t}|\pa_t F-TN-TL|^2(1+r)^{\a+1}+\frac{|\pa\pa_t\phi|^2}{(1+r)^{\a+1}}dxdt\\
&\les \int_{\tau-R}^{\tau}\int_{\Si_t}\frac{|\pa\pa_t\phi|^2}{(1+r)^{1+\a}}dxdt+
D^{2\a}[\pa_t F]_{\tau-R}^{\tau}+D^{2\a}[TN]_{\tau-R}^{\tau}+D^{2\a}[TL]_{\tau-R}^{\tau}\\
&\les\left(E_0\ep^2+C_1\right)(1+\tau)^{-1-\a}.
 \end{split}
\end{equation*}
Hence we can estimate
\begin{equation*}
\begin{split}
 \int_{r\leq 2R}|\pa_{tt}\phi|^2dx\leq 2\int_{r\leq 2R}J^T_\mu[\pa_t\phi]n^\mu d\si
&\les E[\pa_t\phi](\tau-R)+\int_{\tau-R}^{\tau}\int_{r\leq R +t-\tau}|\pa_t F-TN-TL||\pa_{tt}\phi|d\vol\\
&\les\left(E_0\ep^2+C_1\right)(1+\tau_1)^{-1-\a}.
\end{split}
\end{equation*}
Then from ~\eqref{H2phi2R}, we get
\begin{align*}
 \int_{r\leq R}|\pa^2\phi|^2dx&\les\left(E_0\ep^2+C_1\right)(1+\tau_1)^{-1-\a}+
\sum\limits_{j\leq 1}\int_{\tau-R}^{\tau+R}\int_{\Si_t}\frac{|\bar\pa\phi|^2}{(1+r)^{1+\a}}
+|T^j F|^2+|T^jN|^2+|T^jL|^2\\
&\les \left(E_0\ep^2+C_1\right)(1+\tau_1)^{-1-\a}+
\sum\limits_{j\leq 1}D^{\a}[T^j F]_{\tau-R}^{\tau+R}+D^{2\a}[T^j N]_{\tau-R}^{\tau+R}+D^{2\a}[T^j L]_{\tau-R}^{\tau+R}\\
&\les  \left(E_0\ep^2+C_1\right)(1+\tau_1)^{-1-\a}.
\end{align*}
Thus we have proven the lemma for $\tau\geq R$. When $\tau\leq R$,
the finite speed of propagation shows that the solution of
~\eqref{LWAVEEQ} vanishes when $|x|\geq t+R_0$. Thus we can replace
$\tau-R$ with 0 in the above argument. And the lemma still holds.
\end{proof}

A corollary of the above lemma is the following pointwise decay of
the solution in the cylinder $\{r\leq R\}$.

\begin{cor}
 \label{ptdcinc}
Assume that there is a constant $C_1$ such that
\[
 D^{2\a}[F]_{\tau_1}^{\tau_2}+ D^{2\a}[\pa_t F]_{\tau_1}^{\tau_2}\leq C_1(1+\tau_1)^{-1-\a},
  \qquad \forall \tau_2\geq \tau_1\geq 0.
\]
Then for solution $\phi$ of ~\eqref{LWAVEEQ}, we have
\begin{equation*}
|\phi|^2\les \left(C_1+\ep^2E_0\right)(1+\tau)^{-1-\a}, \qquad r\leq
R.
\end{equation*}
\end{cor}
\begin{proof} Using Sobolev embedding and Lemma ~\ref{H2phil}, when $|x|\leq R$, we
can estimate
\begin{align*}
 \phi^2&\les\int_{r\leq R}\sum\limits_{i, j=1}^{3}|\pa_{ij}\phi|^2+\phi^2dx\\
&\les \int_{r\leq R}|\pa^2\phi|^2dx+\int_{\tau}^{\tau+R}\int_{r\leq R}|\pa_t\phi|^2+|\phi|^2dxdt\\
&\les \left(C_1+\ep^2E_0\right)(1+\tau)^{-1-\a},
\end{align*}
where the last step follows from the integrated local energy
inequality ~\eqref{ILE0} restricted to the region $r\leq R$.
\end{proof}

\section{Boostrap Argument}
To solve our nonlinear problem \eqref{THEWAVEEQ}, we use the
standard Picard iteration process. We prove, by a boostrap argument,
that the nonlinear term $D^{2\a}[F]_{\tau_1}^{\tau_2}$ decays, which
leads to the decay of the solution $\phi$. We still denote the
quadratic nonlinearity $A^{\a\b}\pa_\a\phi\pa_\b\phi$ of $F$ in
~\eqref{THEWAVEEQ} as $N(\phi,
\phi)=A^{\mu\nu}\pa_\mu\phi\pa_\nu\phi$, in which the constants
$A^{\mu\nu}$ satisfy the null condition.

\begin{prop}
\label{propboostrap}
 Suppose  $Z^k\Phi$ is $(\delta, \a, t_0, R_1, C_1)-$weak wave
 for all $|k|\leq 3$. Assume
 $$|\pa^2 Z^\b\Phi|\leq C_1,\quad \forall |\b|\leq 1.$$
 Assume the functions $L^\mu(t, x)$, $h^{\mu\nu}(t, x)$ satisfy the conditions in Theorem \ref{maintheorem2}.
If the nonlinearity $F$ in ~\eqref{THEWAVEEQ} satisfies
  \begin{align*}
 &D^{2\a}[Z^\b F]_{\tau_1}^{\tau_2}\leq 2E_0\ep^2 (1+\tau_1)^{-1-\a},
 \quad \forall |\b|\leq 3, \quad \forall \tau_2\geq \tau_1\geq0,\\
 &\int_{r\leq R}|\nabla Z^\b F|^2dx\leq 2E_0 \ep^2(1+\tau)^{-1-\a},
\quad \forall |\b|\leq 1, \quad \forall \tau\geq 0,
\end{align*}
then
\begin{align}
 \label{nullbd}
&D^{2\a}[Z^\b F]_{\tau_1}^{\tau_2}\les E_0^2 \ep^4
(1+\tau_1)^{-1-\a}, \quad \forall |\b|\leq 3, \quad
\forall \tau_2\geq
 \tau_1\geq 0,\\
\label{nullbdnabla} &\int_{r\leq R}|\nabla Z^\b F|^2dx\les E_0^2
\ep^4(1+\tau)^{-1-\a}, \quad \forall |\b|\leq 1, \quad \forall
\tau\geq 0.
\end{align}
\end{prop}
\begin{remark}
\label{remarkpropb} If the given function $\Phi$ is assumed as in
Theorem \ref{maintheorem2}, that is,  $Z^\b\Phi$ is $(\delta, \a,
t_0, R_1, C_1)-$weak wave for all $|\b|\leq 4$ and $|\pa^2Z^\b
\Phi|\leq C_1$ for all $|\b|\leq 2$, then the above proposition
holds if we replace $|\b|\leq 3$, $|\b|\leq 1$ with $|\b|\leq 4$,
$|\b|\leq 2$ respectively. The reason that we formulate the
proposition as above is that three derivatives are the minimum to
close the boostrap argument. Four derivatives is needed to obtain
$C^2$ solution of the equation \eqref{THEWAVEEQ}.
\end{remark}

The proof for Proposition \ref{propboostrap} is quite similar to
that in \cite{yang1}. For completeness, we repeat it here. Since
  higher order nonlinearity decays much better, we only consider the quadratic nonlinearities
  $N(\phi, \phi)$ and $Q(\pa\phi)=h^{\mu\nu}(t, x)\pa_\mu\phi\pa_\nu\phi$.
  First, Lemma ~\ref{nullstructure} and the
assumptions on $h^{\mu\nu}(t, x)$
\[
|Z^\b h^{\mu\nu}|\les (1+r)^{-\frac{3}{2}\a},\quad \forall |\b|\leq
4
\]
 imply that
\begin{equation}
 \label{nullbdsplit}
\begin{split}
D^{2\a}[Z^\b
F]_{\tau_1}^{\tau_2}&\les\sum\limits_{\b_1+\b_2\leq\b}\int_{\tau_1}^{\tau_2}\int_{\Si_\tau}
|N(Z^{\b_1}\phi, Z^{\b_2}\phi)|^2(1+r)^{1+2\a}+|\pa Z^{\b_1}\phi|^2|\pa Z^{\b_2}|^2(1+r)^{1-\a}dxd\tau\\
&\les \sum\limits_{\b_1+\b_2\leq\b}\int_{\tau_1}^{\tau_2}\int_{r\leq
R}|\pa\phi_1|^2|\pa\phi_2|^2dxd\tau
+\sum\limits_{\b_1+\b_2\leq\b}\int_{\tau_1}^{\tau_2}\int_{S_\tau}|N(\phi_1,
\phi_2)|^2 r^{1+2\a}dxd\tau\\
&\qquad+\sum\limits_{\b_1+\b_2\leq\b}\int_{\tau_1}^{\tau_2}\int_{S_\tau}|\pa\phi_1|^2|\pa\phi_2|^2(1+r)^{1-\a}dxd\tau,
\end{split}
\end{equation}
where we denote $\phi_1=Z^{\b_1}\phi$, $\phi_2=Z^{\b_2}\phi$. We
estimate the three integrals on the right hand side of
~\eqref{nullbdsplit} separately. We use elliptic estimates as well
as the extra boostrap assumption ~\eqref{nullbd} to bound the first
integral. Estimates of the second integral rely on the null
structure of $N(\phi_1, \phi_2)$. The third integral follows from
the integrated local energy inequality \eqref{ILE0}.

\subsection{Proof of the ~\eqref{nullbd} in the region $r\leq R$}
When $r\leq R$, we use elliptic estimates to
obtain the pointwise bound of the solution. However since one can only
obtain elliptic estimates in a smaller region, we divide the region
$r\leq R$ into two parts: $r\leq \frac{R}{2}$ and $r\geq
\frac{R}{2}$. In the smaller region $r\leq \frac{R}{2}$, we use elliptic estimates while in the region
$r\geq \frac{R}{2}$, we use Sobolev embedding.

\bigskip

Recall that $\phi_1=Z^{\b_1}\phi$, $\phi_2=Z^{\b_2}\phi$,
$|\b_1|+|\b_2|\leq 3$. Without loss of generality, assume
$|\b_1|\leq |\b_2|$. In particular we have $|\b_1|\leq 1$. For
$r\leq \frac{R}{2}$, we claim that
\begin{equation}
 \label{ptbdpaphi}
|\pa\phi_1|^2\les E_0\ep^2(1+\tau)^{-1-\a},\quad r\leq \frac{R}{2}.
\end{equation}
We first verify ~\eqref{ptbdpaphi} for $\pa_t\phi_1$. By ~\eqref{waveeqomktj},  $\pa_t\phi_1$ satisfies the following equation
\[
 \Box (\pa_t\phi_1)+N(\Phi, \pa_t\phi_1)+L(\pa_t \phi_1)=F_1.
\]
Since $|\b_1+(0, 2)|\leq 3$, estimates ~\eqref{inductionD} imply
that
\[
 D^{2\a}[F_1]_{\tau_1}^{\tau_2}+ D^{2\a}[\pa_t F_1]_{\tau_1}^{\tau_2}\les E_0\ep^2
 (1+\tau_1)^{-1-\a}.
\]
Thus by Corollary ~\ref{ptdcinc}, we have
\[
 |\pa_t\phi_1|^2\les E_0\ep^2 (1+\tau)^{-1-\a}, \qquad r\leq
 \frac{R}{2}.
\]
For $\nabla \phi_1$, notice that $\phi_1=Z^{\b_1}\phi$, $|\b_1|\leq
1$ and $|\pa^2 Z^{\b_1}\Phi|\les 1$, $|\pa Z^{\b_1}L^\mu|\les1$.
Using elliptic estimates and Lemma \ref{H2phil}, we have for
$|x|\leq \frac{R}{2}$
\begin{equation}
\label{Caest}
\begin{split}
\|\nabla\phi_1\|_{C^\f12(B_{\f12 R})}^2&\les \int_{r\leq \frac{R}{2}}\sum\limits_{i, j=1}^{3}|\pa_{ij}\nabla\phi_1|^2+|\nabla\phi_1|^2dx\\
&\les \int_{r\leq R}|\nabla\Delta\phi_1|^2+ |\nabla\phi_1|^2dx\\
&\les \int_{r\leq R}|\nabla\left(\pa_{tt}\phi_1+Z^{\b_1}F-Z^{\b_1}N(\Phi, \phi)-Z^{\b_1}L(\pa\phi)\right)|^2+ |\nabla\phi_1|^2dx\\
&\les E[\pa_{tt}\phi_1](\tau)+E[\phi_1](\tau)+\int_{r\leq R}|\nabla
Z^{\b_1}F|^2+|\nabla^2 \phi_1|^2+|\nabla^2\phi|^2+|\nabla
\phi_1|^2+|\nabla\phi|^2dx\\
&\les E_0\ep^2(1+\tau_1)^{-1-\a},
\end{split}
\end{equation}
where $B_{\f12 R}$ denotes the ball $\{r\leq \f12 R\}$ in $\mathbb{R}^3$.
 Hence we have proven ~\eqref{ptbdpaphi}, which implies that
\begin{align*}
\int_{\tau_1}^{\tau_2}\int_{r\leq \frac{R}{2}}|\pa\phi_1|^2|\pa\phi_2|^2dx &\les\int_{\tau_1}^{\tau_2}(1+\tau)^{-1-\a}E_0 \ep^2 \int_{r\leq\frac{
R}{2}}|\pa\phi_2|^2dxd\tau\\
&\les E_0 \ep^2\int_{\tau_1}^{\tau_2}(1+\tau)^{-1-\a}E[\phi_2](\tau)d\tau\\
&\les E_0^2 \ep^4(1+\tau)^{-1-\a}.
\end{align*}

\bigskip

In the region $\frac{R}{2}\leq r\leq R$, we use the angular momentum
$\Om$. By Sobolev embedding on the unit sphere, we have
\begin{equation}
 \label{SemSphere}
\int_{\om}|\pa\phi_1|^2\cdot|\pa\phi_2|^2d\om\les\sum\limits_{1'}\int_{\om}|\pa\phi_{1'}|^2
d\om\cdot\int_{\om}|\pa\phi_{2}|^2d\om,
\end{equation}
where we still denote $\phi_{1'}=Z^{\b_{1'}}\phi$ for $\b_{1'}\leq
\b_1+(2, 0)$.
 Notice that $|\b_{1'}|+|\b_2|\leq 3+2=5$. Without loss of generality, we assume $|\b_2|\leq 2$. Thus
  by Lemma ~\ref{H2phil}, we have
\begin{align}
\label{SemrR} &\int_{\om}|\pa\phi_{2}|^2d\om\les
\int\limits_{\frac{R}{2}\leq r\leq R}|\pa\phi_{2}|^2+
|\pa_r\pa\phi_{2}|^2dx\les E[\phi_{2}](\tau)+\int_{r\leq
R}|\pa^2\phi_{2}|^2dx\les (1+\tau)^{-1-\a}E_0 \ep^2,
\end{align}
where we have used ~\eqref{inductionD}, ~\eqref{waveeqomktj} and the
assumption $|\b_2|\leq 2$ to verify the conditions in Lemma
~\ref{H2phil}. Since $|\b_{1'}|\leq 3$, we can show that
\begin{align*}
\int_{\tau_1}^{\tau_2}\int_{\f12 R\leq r\leq R}|\pa\phi_1|^2|\pa\phi_2|^2dxd\tau &\les\int_{\tau_1}^{\tau_2}\int_{\frac {R}{2}}^{R}\int_{\om}|\pa\phi_1|^2|\pa\phi_2|^2d\om \quad r^2drd\tau\\
&\les\sum\limits_{1'} \int_{\tau_1}^{\tau_2}\int_{\frac {R}{2}}^{R}\int_{\om}|\pa\phi_{1'}|^2d\om\int_{\om}|\pa\phi_{2}|^2d\om\quad r^2drd\tau\\
&\les \sum\limits_{1'}\int_{\tau_1}^{\tau_2}(1+\tau)^{-1-\a}E_0 \ep^2 \int_{r\leq R}|\pa\phi_{1'}|^2dxd\tau\\
&\les E_0^2 \ep^4(1+\tau)^{-1-\a}.
\end{align*}
 Summarizing, we have shown
\[
\sum\limits_{\b_1+\b_2\leq\b}\int_{\tau_1}^{\tau_2}\int_{r\leq
R}|\pa\phi_1|^2|\pa\phi_2|^2dxd\tau\les E_0^2 \ep^4(1+\tau)^{-1-\a}.
\]
\begin{remark}
 We remark here that ~\eqref{SemrR} is only true when $r$ is bigger that a constant. That is why we need to distinguish the two cases $r\leq \f12 R$ and $r\geq \f12 R$.
\end{remark}

\subsection{Proof of ~\eqref{nullbdnabla}}
Note that when $|\b|\leq 1$ and $\b_1+\b_2=\b$, we have $\b_1=0$ or
$\b_2=0$. By~\eqref{SemSphere}, we have
\begin{equation*}
\begin{split}
 \int_{\om}|\nabla Z^\b F|^2d\om &\les\int_{\om}|\pa Z^\b\phi|^2\cdot|\pa^2\phi|^2+|\pa\phi|^2|\pa^2 Z^\b\phi|^2d\om\\
&\les\sum\limits_{|\b'|\leq 2}\int_{\om}|\pa
Z^\b\phi|^2d\om\cdot\int_{\om}|\pa^2
Z^{\b'}\phi|^2d\om+\int_{\om}|\pa
Z^{\b'}\phi|^2d\om\cdot\int_{\om}|\pa^2 Z^{\b}\phi|^2d\om,
\end{split}
\end{equation*}
where as pointed out previously, we only have to consider the
quadratic nonlinearities $N(\phi, \phi)$, $Q(\pa\phi)$. For $r\leq
\f12 R$, the inequality ~\eqref{ptbdpaphi} shows that
\[
 |\pa Z^\b\phi|\les \ep^2 E_0(1+\tau)^{-1-\a}, \quad \forall |\b|\leq
 1.
\]
For $\f12 R\leq r\leq R$, the inequality ~\eqref{SemrR} implies that
$$\int_{\om}|\pa Z^{\b'}\phi|^2d\om\les \ep^2 E_0(1+\tau)^{-1-\a}, \quad \forall |\b'|\leq
2.
$$
On the other hand, using Lemma ~\ref{H2phil}, we obtain
$$
\int_{r\leq R}|\pa^2 Z^{\b'}\phi|^2dx\les \ep^2
E_0(1+\tau)^{-1-\a},\quad \forall|\b'|\leq 2.
$$
Therefore, for all $|\b|\leq 1$, we can estimate
\begin{align*}
 \int_{r\leq R}|\nabla Z^\b F|^2dx&\les\ep^2 E_0(1+\tau)^{-1-\a}\sum\limits_{|\b'|\leq 2}
 \int_{r\leq R}|\pa^2 Z^{\b'}\phi|^2dx\les
 E_0^2\ep^4(1+\tau)^{-1-\a}.
\end{align*}
Hence we have proven ~\eqref{nullbdnabla}.

\subsection{Proof of ~\eqref{nullbd} in the region $r\geq R$}
We first consider the quadratic term $N(\phi_1, \phi_2)$. The
p-weighted energy inequality is about $\psi=r\phi$ instead of
$\phi$. For this reason, we expand $N(\phi_1, \phi_2)$ in terms of
$\psi$.
\begin{lem}
\label{lemnullform} Suppose $N(\phi_1,
\phi_2)=A^{\a\b}\pa_\a\phi_1\pa_\b\phi_2$ with constants $A^{\a\b}$ satisfying
the null condition. Then
\begin{equation}
\label{nullform}
 r^4|N(\phi_1, \phi_2)|^2\les\phi_1^2\phi_2^2+\phi_1 ^2\cdot
r^2\pa_r^2\phi_2+|\nabb\psi_1|^2|\nabb\psi_2|^2+|\pa_v\psi_1|^2|\pa_u\psi_2|^2+|\pa_u\psi_1|^2|\pa_v\psi_2|^2,
\end{equation}
 where $v=\frac{t+r}{2}, u=\frac{t-r}{2}$.
\end{lem}
\begin{proof} In fact, notice that
$$r^2 N(\phi_1, \phi_2)=\phi_1\phi_2+r(\phi_1\phi_2)_r+N(\psi_1, \psi_2)
$$
and
$$|N(\psi_1, \psi_2)|\les
|\pa_v\psi_1|\cdot|\pa_u\psi_2|+|\pa_u\psi_1|\cdot|\pa_v\psi_2|+|\nabb\psi_1|\cdot|\nabb\psi_2|.
$$
Hence the lemma holds.
\end{proof}

To estimate the second term in ~\eqref{nullbdsplit}, it suffices to
handle the terms on the right hand side of ~\eqref{nullform}. We
estimate the first three terms in a uniform way.
   Let $\Phi_1$ be $\phi_1$ or $\nabb\psi_1$; $\Phi_2$ be $\phi_2$, $r\pa_r\phi_2$ and $\nabb\psi_2$ respectively. Recall that $\phi_1=Z^{\b_1} \phi$, $\phi_2=Z^{\b_2}\phi$, $|\b_1|+|\b_2|\leq 3$. Using Sobolev embedding on the unit sphere, we have
\begin{equation*}
 \int_{\om}|\Phi_1|^2|\Phi_2|^2d\om\les\sum\limits_{1',
2'}\int_{\om}|\Phi_{1'}|^2d\om\cdot\int_{\om}|\Phi_{2'}|^2d\om,
\end{equation*}
where we let
\begin{equation}
\label{bs}
\begin{cases}
 \b_{1'}\leq \b_1+(2, 0), \quad\b_{2'}=\b_2,\quad \textnormal{if } |\b_1|\leq 1,\\
\b_{1'}=\b_1$, \quad $\b_{2'}\leq \b_2+(2, 0),\quad \textnormal{if }
|\b_2|\leq 1.
\end{cases}
\end{equation}
In particular $|\b_{1'}|+|\b_{2'}|\leq 5$. For the third case when
$\Phi_1=\nabb\psi_1$, $\Phi_2=\nabb\psi_2$, without loss of
generality, we assume $|\b_{1'}|\leq 2$. Since
$\nabb\psi_1=\Om\phi_1$, $\Phi_{1'}$ can always be written as
$Z^\b\phi$ for some $|\b|\leq 3$. Thus
 by Corollary ~\ref{ptdcoutc}, we have
$$
r^2\int_{\om}|\Phi_{1'}|^2d\om\les\ep^2 E_0,\qquad r\geq R.
$$
Recall that $\Phi_{2'}=\phi_{2'}$, $r\pa_r\phi_{2'}$ or
$\nabb\psi_{2'}$. We always have
\[
 \frac{|\Phi_{2'}|^2}{(1+r)^{3+\a}}\les
 \frac{|\bar\pa\phi_{2'}|^2}{(1+r)^{1+\a}}.
\]
Then the integrated energy inequality ~\eqref{ILE0} implies that
\begin{equation*}
\begin{split}
\int_{\tau_1}^{\tau_2}\int_{S_\tau}r^{2\a-3}\Phi_{1}^2\Phi_{2}^2 d\vol&=\int_{\tau_1}^{\tau_2}\int_{v_\tau}^{\infty}\int_{\om}r^{2\a-1}\Phi_{1}^2\Phi_{2}^2 dvd\om d\tau\\
&\les\sum\limits_{1', 2'} \int_{\tau_1}^{\tau_2}\int_{v_\tau}^{\infty} r^{2\a-3}r^2
\int_{\om}|\Phi_{1'}|^2d\om\int_{\om}|\Phi_{2'}|^2d\om dvd\tau\\
&\les\ep^2 E_0\sum\limits_{1', 2'}\int_{\tau_1}^{\tau_2}\int_{S_\tau}\frac{|\Phi_{2'}|^2}{r^{3-2\a}}dvd\om d\tau\\
&\les\ep^2 E_0 \sum\limits_{1', 2'}\int_{\tau_1}^{\tau_2}\int_{S_\tau}\frac{|\Phi_{2'}|^2}{(1+r)^{3+\a}}d\vol\\
&\les \ep^4 E_0^2(1+\tau_1)^{-1-\a},
\end{split}
\end{equation*}
where we recall that $\a\leq \frac{1}{4}$. We hence have estimated
the first three terms in \eqref{nullform}.

\bigskip

It remains to handle the last two terms
$|\pa_v\psi_1|^2|\pa_u\psi_2|^2$, $|\pa_u\psi_1|^2|\pa_v\psi_2|^2$.
Since they are symmetric,
it suffices to consider $|\pa_v\psi_1|^2|\pa_u\psi_2|^2$. Recall
 that $\psi_1=rZ^{\b_1}\phi$, $\psi_2=rZ^{\b_2}\phi$, $|\b_1|+|\b_2|\leq 3$. Define $\b_{1'}, \b_{2'}$ as in ~\eqref{bs}.
 In particular $|\b_{1'}|+|\b_{2'}|\leq 5$. We have two cases according to $\b_{i'}$, $i=1, 2$.

\bigskip

We first consider the case when $|\b_{1'}|\leq 2$. The idea is that
we bound $|\pa_v\psi_1|$ uniformly and then control
$|\pa_u\psi_2|^2$ by the energy flux through the null hypersurface
$v=constant$. The following lemma shows that the energy flux through
$v=constant$ is bounded.
\begin{lem}
\label{crossnullen}
 Consider the region $D=[u_1, u_2]\times[v_1, \infty)\subset S_\tau \times [\tau_1, \tau_2]$. Under the conditions of proposition ~\ref{propboostrap}, we have the energy flux estimate through the hypersurface $v=const$
$$\int_{u_1}^{u_2}\int_{\om}(\pa_u\psi_2)^2dud\om \les \ep^2
E_0(1+\tau_1)^{-1-\a},
$$
where $\psi_2=r\phi_2=rZ^{\b_2}\phi$.
\end{lem}
\begin{proof} Back to the energy equation ~\eqref{energyeq}, take $X=T$ on the region $D$. We have
\begin{align*}
\int_{u_1}^{u_2}J_\mu^T[\phi_2]n^{\mu}d\si + \iint\limits_{v\geq v_1, u=u_1 }J_\mu^T[\phi_2]n^{\mu}d\si&=\iint\limits_{v\geq v_1,u=u_2 }J_\mu^T[\phi_2]n^{\mu}d\si+\int_{I_{\tau_1}^{\tau_2}}J_\mu^T[\phi_2]n^{\mu}d\si\\
&\quad+ \int_D \Box\phi_2\cdot \pa_t\phi_2d\vol.
\end{align*}
Using the estimates ~\eqref{inductionD}, we conclude that
\[
 D^{2\a}[\Box\phi_2]_{\tau_1}^{\tau_2}\les E_0\ep^2(1+r)^{-1-\a}.
\]
Then by the integrated local energy inequality \eqref{ILE0} and the
energy inequality ~\eqref{eb}, we can show that
\begin{align*}
\int_{u_1}^{u_2}\int_{\om}r^2(\pa_u\phi_2)^2d\om du &\leq
2\int_{u_1}^{u_2}J_\mu^T[\phi_2]n^{\mu}d\si\les E_0\ep^2
(1+\tau_1)^{-1-\a},
\end{align*}
where notice that $D\subset S_\tau\times[\tau_1, \tau_2]$. Thus by
Corollary ~\ref{ptdcoutc}, we get
\begin{align*}
 \int_{u_1}^{u_2}\int_{\om}(\pa_u\psi_2)^2d\om du &= \int_{u_1}^{u_2}\int_{\om}r^2(\pa_u\phi_2)^2dud\om +
 \left.\int_{\om}r\phi_2^2d\om \right|_{u_1}^{u_2}\les
 E_0\ep^2(1+\tau_1)^{-1-\a}.
\end{align*}
\end{proof}

\bigskip

We continue our proof of \eqref{nullbd} when $|\b_{1'}|\leq 2$.
Lemma \ref{crossnullen} and Sobolev embedding on the unit sphere
show that
\begin{equation*}
\begin{split}
&\int_{\tau_1}^{\tau_2}\int_{S_\tau}r^{2\a-3}|\pa_v\psi_1|^2|\pa_u\psi_2|^2
d\vol\\
&=\int_{v_{\tau_1}}^{\infty}\int_{u_{\tau_1}}^{u(v)}r^{2\a-1}\int_{\om}|\pa_u\psi_2|^2|\pa_v\psi_1|^2d\om du dv\\
 &\leq \int_{v_{\tau_1}}^{\infty}\int_{u_{\tau_1}}^{u(v)}|\pa_u\psi_{2'}|^2 d\om du\cdot \sup\limits_{u}\int_{\om}r^{2\a-1}|\pa_v\psi_{1'}|^2d\om\quad dv\\
& \les \ep^2
E_0(1+\tau_1)^{-1-\a}\int_{v_{\tau_1}}^{\infty}\sup\limits_{u}\int_{\om}r^{2\a-1}|\pa_v\psi_{1'}|^2d\om
dv\\
&\les \ep^2
E_0(1+\tau_1)^{-1-\a}\int_{v_{\tau_1}}^{\infty}\sup\limits_{u}\int_{\om}r^{2\a}|\pa_v\psi_{1'}|^2d\om
dv,
\end{split}
\end{equation*}
where $\b_{1'}, \b_{2'}$ are defined as in ~\eqref{bs}. Now, for all
$ u\in [u_{\tau_1}, u(v)]$, we have
\begin{align*}
r^{2\a}(\pa_v\psi_{1'})^2&\les\left.r^{2\a}(\pa_v\psi_{1'})^2\right|_{u=u_1} + \int_{u_1}^{u_2}r^{2\a}(\pa_v\psi_{1'})^2du \\
&\quad\quad+ \int_{u_1}^{u_2}r^{2\a}(\pa_u\pa_v\psi_{1'})^2du + \int_{u_1}^{u_2}r^{2\a-1}(\pa_v\psi_{1'})^2du\\
&\les \left.r^{2\a}(\pa_v\psi_{1'})^2\right|_{u=u_1} + \int_{u_1}^{u_2}r^{2\a}(\pa_v\psi_{1'})^2du\\
& \quad\quad+ \int_{u_1}^{u_2}r^{2\a}(\lap\psi_{1'})^2+
r^{2\a+2}|Z^{\b_{1'}}(F-N-L)|^2du,
\end{align*}
where we use the wave equation ~\eqref{waveqpsi} in the last step
and $u_1=u_\tau, u_2=u(v)$. Integrate on the unit sphere. We obtain
\begin{align*}
 &\int_{v_{\tau_1}}^{\infty}\sup\limits_{u}\int_{\om}r^{2\a}|\pa_v\psi_{1'}|^2 d\om dv
\les \int_{S_{\tau_{1}}}r^{2\a}(\pa_v\psi_{1'})^2dvd\om\\
& +\int_{\tau_{1}}^{\tau_2}\int_{S_\tau}r^{2\a}(\pa_v\psi_{1'})^2 +
(\nabb\Om\phi_{1'})^2r^{2\a} +r^{2\a+2}|Z^{\b_{1'}}(F-N-L)|^2 dvd\om
d\tau,
\end{align*}
where note that $\nabb=\frac{\Om}{r}$.

We claim that the above inequality can be bounded by $\ep^2 E_0$(up
to a constant). In fact, the first term can be bounded by $\ep^2
E_0$ by ~\eqref{pwe1a}; the second term can be bounded by
 $\ep^2 E_0$  by ~\eqref{pwe1ai}; the third term can be controlled by $(1+\tau_1)^{-1-\a}\ep^2 E_0$
 by the integrated local energy inequality ~\eqref{ILE0}(notice that $|\b_{1'}|+1\leq3$
 and $2\a\leq 1-\a$ as $\a\leq \frac{1}{4}$); the last term can be
 estimated as
\[
 D^{2\a}[Z^{\b_{1'}}F]_{\tau_1}^{\tau_2}+D^{2\a}[Z^{\b_{1'}}N]_{\tau_1}^{\tau_2}+D^{2\a}
 [Z^{\b_{1'}}L]_{\tau_1}^{\tau_2}\les E_0\ep^2(1+\tau_1)^{-1-\a},
\]
where we use the inequality ~\eqref{inductionD}. Summarizing, we
have shown
$$\int_{v_{\tau_1}}^{\infty}\sup\limits_{u}\int_{\om}r^{2\a}|\pa_v\psi_1|^2 d\om dv\les \ep^2
E_0.
$$
In particular, for fixed $r\geq R$, we have
\begin{equation}
\label{pavphipt} \int_{t_2}^{t_1}\int_{\om}|\pa_v\psi_1|^2(t, r,
\om)d\om dt\les \ep^2 E_0,\quad \psi_1=r Z^\b\phi, \quad |\b|\leq 2.
\end{equation}
Therefore
\begin{equation*}
\int_{\tau_1}^{\tau_2}r^{2\a-3}\int_{S_\tau}|\pa_v\psi_1|^2|\pa_u\psi_2|^2
d\vol\les E_0^2\ep^4(1+\tau_1)^{-1-\a}.
\end{equation*}

\bigskip

\textbf{When} $|\b_{2'}|\leq 2$, we control $|\pa_v\psi_1|^2$ by the
energy and bound $\pa_u\psi_2$ uniformly. Similarly, using
~\eqref{pwe1a}, ~\eqref{pwe1ai} and Sovolev embedding, we obtain
\begin{align}
 \notag
 &\int_{\tau_1}^{\tau_2}\int_{S_\tau}r^{2\a-1}|\pa_v\psi_1|^2|\pa_u\psi_2|^2 dvd\om d\tau\\
\label{pro3sup1}
&\les \int_{\tau_1}^{\tau_2}\int_{S_\tau}r^{2\a-1}|\pa_v\psi_{1'}|^2 \cdot\int_{\om}|\pa_u\psi_{2'}|^2d\om \quad dvd\om d\tau\\
\notag
&\les \int_{\tau_1}^{\tau_2}\int_{S_\tau}r^{3\a-1}|\pa_v\psi_{1'}|^2 \cdot r^{-\a}\int_{\om}|\pa_u\psi_{2'}|^2d\om \quad dvd\om d\tau\\
\notag &\les \ep^2 E_0\int_{\tau_1}^{\tau_2}
\sup\limits_{v}r^{-\a}\int_{\om}|\pa_u\psi_{2'}|^2d\om d\tau,
\end{align}
where $\b_{1'}, \b_{2'}$ are defined as in ~\eqref{bs}. For all $v$,
we have
\begin{align*}
 r^{-\a}(\pa_u\psi_{2'})^2&\les \left.r^{-\a}(\pa_u\psi_{2'})^2\right|_{v=v_{\tau_2}}+\left|\int_{v}^{v_{\tau_2}}r^{-1-\a}|\pa_u\psi_{2'}|^2dv\right|\\
&\quad\quad +2\left|\int_{v}^{v_{\tau_2}}r^{-\a}|\pa_u\psi_{2'}\cdot\pa_v\pa_u\psi_{2'}|dv\right|\\
&\les\left.r^{-\a}(\pa_u\psi_{2'})^2\right|_{v=v_{\tau_2}}+\int_{v_\tau}^{\infty}r^{-1-\a}|\pa_u\psi_{2'}|^2dv\\
&\quad\quad + \int_{v_\tau}^{\infty}r^{-1-\a}(\pa_u\psi_{2'})^2dv + \int_{v_\tau}^{\infty}r^{1-\a}(\pa_v\pa_u\psi_{2'})^2dv\\
&\les\left.r^{-\a}(\pa_u\psi_{2'})^2\right|_{v=v_{\tau_2}}+ \int_{v_\tau}^{\infty}\frac{(\pa_u\psi_{2'})^2}{r^{1+\a}}dv\\
&\quad\quad +\int_{v_\tau}^{\infty}r^{1-\a}(\lap\psi_{2'})^2dv+
\int_{v_\tau}^{\infty}r^{3-\a}|Z^{\b_{2'}}(F-N-L)|^2dv,
\end{align*}
where $v_{\tau_2}=\frac{R+\tau_2}{2}$. Integrate on the unit
sphere. We get
\begin{align*}
 &\int_{\tau_1}^{\tau_2}
\sup\limits_{v}r^{-\a}\int_{\om}|\pa_u\psi_{2'}|^2d\om d\tau\les
\int_{\tau_1}^{\tau_2}\int_{\om}\left.r^{-\a}(\pa_u\psi_{2'})^2\right|_{v=v_{\tau_2}}d\tau \\
&\quad\quad+\int_{\tau_1}^{\tau_2}\int_{S_\tau}\frac{(\pa_u\psi_{2'})^2}{r^{1+\a}}+r^{1-\a}
(\nabb\Om\phi_{2'})^2+r^{3-\a}|Z^{\b_{2'}}(F-N-L)|^2dvd\om d\tau.
\end{align*}
We claim that it can be bounded by  $(1+\tau_1)^{-1-\a}\ep^2
E_0$ up to a constant. In fact, the first term can be estimated as follows
$$\int_{\tau_1}^{\tau_2}\int_{\om}\left.r^{-\a}(\pa_u\psi_{2'})^2\right|_{v=v_{\tau_2}}d\tau=
\int_{u_{\tau_1}}^{u_{\tau_2}}\int_{\om}r^{-\a}|\pa_u\psi_{2'}|^2d\om
du\leq\int_{u_{\tau_1}}^{u_{\tau_2}}\int_{\om}|\pa_u\psi_{2'}|^2d\om
du \les(1+\tau_1)^{-1-\a}\ep^2 E_0
$$
by Lemma ~\ref{crossnullen}; the second and third term can also be
controlled by $(1+\tau_1)^{-1-\a}\ep^2 E_0$ by the integrated local
energy estimates ~\eqref{ILE0}(notice that $|\b_{2'}|\leq 2$); the
last term can be estimated as follows
\[
 D^{2\a}[Z^{\b_{1'}}F]_{\tau_1}^{\tau_2}
+D^{2\a}[Z^{\b_{1'}}N]_{\tau_1}^{\tau_2}+
D^{2\a}[Z^{\b_{1'}}L]_{\tau_1}^{\tau_2}\les
E_0\ep^2(1+\tau_1)^{-1-\a}
\]
by the inequality ~\eqref{inductionD}. Hence
\begin{equation}
\label{dupsi}
 \int_{\tau_1}^{\tau_2}  \sup\limits_{v}r^{-\a}\int_{\om}|\pa_u\psi_{2'}|^2d\om d\tau\les
 E_0\ep^2(1+\tau_1)^{-1-\a}.
\end{equation}
Plug this into ~\eqref{pro3sup1}. We obtain
$$\int_{\tau_1}^{\tau_2}\int_{S_\tau}r^{2\a-3}|\pa_v\psi_1|^2|\pa_u\psi_2|^2 d\vol\les \ep^4
E_0^2(1+\tau_1)^{-1-\a}.
$$
Therefore using lemma \ref{lemnullform}, we have shown that
\[
 \int_{\tau_1}^{\tau_2}\int_{S_\tau}|N(\phi_1, \phi_2)|^2 r^{1+2\a}d\vol\les \ep^4
 E_0^2(1+\tau_1)^{-1-\a}.
\]

\bigskip

To show \eqref{nullbd}, it remains to estimate the quadratic term
$Z^\b Q(\pa\phi)$. By \eqref{nullbdsplit}, it suffices to consider
the third integral in \eqref{nullbdsplit}. Notice that
\[
|\pa\phi_1||\pa\phi_2|\les
r^{-2}|\pa_v\psi_1||\pa_u\psi_2|+\sum\limits_{\substack{|\b_{1'}|+|\b_{2'}|\leq
4\\\b_{1'},\b_{2'}\leq 3}}|Z^{\b_{1'}}\phi||\pa
Z^{\b_{2'}}\phi|+(|Z\phi|+|\pa_v\phi|)|\pa Z^3\phi|.
\]
The first term has already been estimated considering that
$r^{1-\a}\leq r^{1+2\a}$(on $S_\tau$, $r\geq R\geq 1$). For the
second term, if $|\b_{1'}|\leq 1$ or $|\b_{2'}|\leq 1$, then using
Sobolev inequality on the unit sphere, we have
\[
\int_{\om}|\phi_{1'}|^2|\pa
\phi_{2'}|^2d\om\les\sum\limits_{|\b_{1''}|, |\b_{2''}|\leq
3}\int_{\om}|\phi_{1''}|^2d\om\cdot
\int_{\om}|\pa\phi_{2''}|^2d\om\les r^{-2}\ep^2E_0\sum\limits_{
|\b_{2''}|\leq 3}\int_{\om}|\pa \phi_{2''}|^2d\om
\]
by \eqref{phi2bd} and \eqref{pwe1a}. Then using the integrated local
energy inequality \eqref{ILE0}, we obtain
\begin{align*}
\int_{\tau_1}^{\tau_2}\int_{S_\tau}|\pa\phi_1|^2|\pa\phi_2|^2
(1+r)^{1-\a}d\vol\les \ep^2
E_0\int_{\tau_1}^{\tau_2}\int_{S_\tau}\frac{|\pa\phi_{2''}|^2}{(1+r)^{1+\a}}d\vol\les
\ep^4 E_0^2(1+\tau_1)^{-1-\a}.
\end{align*}
If both $|\b_{1'}|\geq 2$, $|\b_{2'}|\geq 2$, recall that
$|\b_{1'}|+|\b_{2'}|\leq|\b|+1\leq 4$. We conclude that
$|\b_{1'}|=|\b_{2'}|=2$, for which we use the embedding
\[
\int_{\om}|Z^2\phi|^2|\pa Z^2\phi|^2d\om\les
\|Z^2\phi\|_{H^1(S^2)}^2\|\pa Z^2\phi\|_{H^1(S^2)}^2\les r^{-2}\ep^2
E_0 \|\pa Z^2\phi\|_{H^1(S^2)}^2.
\]
Hence
\begin{align*}
\int_{\tau_1}^{\tau_2}\int_{S_\tau}|\pa\phi_1|^2|\pa\phi_2|^2
(1+r)^{1-\a}d\vol\les \sum\limits_{|\b|\leq 1}\ep^2
E_0\int_{\tau_1}^{\tau_2}\int_{S_\tau}\frac{|\pa Z^2
\Om^\b\phi|^2}{(1+r)^{1+\a}}d\vol\les \ep^4 E_0^2(1+\tau_1)^{-1-\a}.
\end{align*}
For the third term $(|Z\phi|+|\pa_v\phi|)|\pa Z^3\phi|$, using the integrated local energy inequality,
it suffices to show that $$(|Z\phi|+|\pa_v\phi|)^2(1+r)^{2}\les
\ep^2 E_0.$$ In fact, by Corollary \ref{ptdcoutc} and Corollary
\ref{ptdcinc}, we have
\[
(1+r)^2|Z\phi|^2\les \sum\limits_{|\b|\leq 3}\int_{\om}|Z^\b
\phi|d\om\les \ep ^2 E_0.
\]
To bound $|\pa_v\phi|$, notice that $r\geq R$. Inequality
\eqref{pavphipt} implies that
\[
\sum\limits_{|\b|\leq
2}\int_{t_2}^{t_1}\int_{\om}|Z^\b\pa_v\psi|d\om dt\les \ep^2
E_0,\quad \psi=r \phi.
\]
Then using Sobolev embedding on $[t_1, t_2]\times S^2$, we obtain
\[
r^2|\pa_v\phi|^2\leq |\pa_v\psi|^2+|\phi|^2\les\ep^2 E_0.
\]
Therefore
\begin{equation*}
(1+r)^2|\pa \phi|^2\les (1+r)^2(|Z\phi|^2+|\pa_v\phi|^2)\les \ep^2
E_0.
\end{equation*}
 In sum, ~\eqref{nullbd} follows from ~\eqref{nullbdsplit}. And have
 proven Proposition ~\ref{propboostrap}.
\begin{remark}
 We in fact can show that
\[
 D^{1+\a}[Z^\b F]_{\tau_1}^{\tau_2}\les
 E_0^2\ep^4(1+\tau_1)^{-1-\a}.
\]
However, it is sufficient to consider $D^{2\a}[Z^\b
F]_{\tau_1}^{\tau_2}$ in order to close the boostrap argument.
\end{remark}

\section{Proof of the Main Theorems }
We used the foliation $\Si_\tau$, part of which is null, in the
previous argument. However, we do not have a local existence result
with respect to the foliation $\Si_\tau$. To solve the nonlinear
equation \eqref{THEWAVEEQ}, we use the standard Picard iteration
process. Take $\phi_{-1}(t, x)=0$. We solve the following linear
wave equation recursively
\begin{equation}
\label{iteration}
\begin{cases}
 \Box\phi_{n+1}+N(\Phi,\phi_{n+1})+L(\pa\phi_{n+1})=F(\pa\phi_n), \\
 \phi_{n+1}(0,x)=\ep \phi_0(x), \pa_t\phi_{n+1}(0,x)=\ep \phi_1(x).
\end{cases}
\end{equation}
Now suppose the implicit constant in Proposition ~\ref{propboostrap}
is $C_1$,
 which, according to our notation, depends only on $R$, $\a$, $t_0$, $C_0$, $A^{\a\b}$, $B^{\a\b}$. Set
$$\ep_0=\frac{1}{\sqrt{C_1 E_0}}.
$$
Then for all $\ep\leq \ep_0$, we have
\[
 C_1\ep^4 E_0^2\leq \ep^2 E_0.
\]
Thus by the continuity of $F(\pa\phi_n)$, we in fact have shown that
the nonlinear term $F$ satisfies
\begin{equation*}
 D^{2\a}[Z^{\b} F(\pa\phi_n)]_{\tau_1}^{\tau_2}\leq C_1E_0^2 \ep^4
 (1+\tau_1)^{-1-\a}\leq E_0\ep^2(1+\tau_1)^{-1-\a}, \quad \forall |\b|\leq 4, \quad \forall \tau_2\geq
 \tau_1\geq 0.
\end{equation*}
 Therefore, Proposition ~\ref{energydecay} implies that
\begin{equation*}
 E[Z^\b\phi_n](\tau)\les E_0 \ep^2 (1+\tau)^{-1-\a},\quad \forall n,\quad \forall |\b|\leq
 4.
\end{equation*}
After using Sobolev embedding on the unit sphere, Corollary
~\ref{ptdcoutc} and Corollary \ref{ptdcinc} indicate that
\begin{equation*}
\begin{split}
 &|Z^\b\phi_n|\les\sqrt{E_0}\ep (1+r)^{-\f12}(1+|t-r+R|)^{-\f12-\frac{1}{2}\a},\quad \forall |\b|\leq 2,\\
& |\phi_n|\les\sqrt{E_0}\ep (1+r)^{-1}.
\end{split}
\end{equation*}
We also need to show that $\phi_n$ is uniformly bounded in $C^2$. We first show that $\phi_n$ is bounded in $C^1$.
When $r\leq \f12 R$, estimates \eqref{ptbdpaphi} implies that
\[
 |\pa Z^k\phi_n|^2\les E_0\ep^2(1+\tau)^{-1-\a},\quad \forall k\leq 2.
\]
Here we have to point out that although there $|\b_1|\leq 1$(due to the assumption $|\b_1|+|\b_2|\leq 3$), the estimate holds for
$|\b_1|\leq 2$ if we assume $|\b_1|+|\b_2|\leq 4$, see Remark \ref{remarkpropb}. When $\f12 R\leq r\leq R$, using \eqref{SemrR} and
Sobolev embedding on the unit sphere, we obtain the same estimates as above. For $r\geq
R$, the inequality \eqref{dupsi} implies that
$$\int_{\tau_1}^{\tau_2}  \sup\limits_{v}r^{-\a}\int_{\om}|\pa_u\psi_n|^2d\om d\tau
\les(1+\tau_1)^{-1-\a}\ep^2E_0,\quad \psi_n=rZ^\b\phi_n,\quad |\b|\leq 3.
$$
Using Sobolev embedding on $S^2\times[\tau_1, \tau_2]$, we obtain
$$|r\pa_u Z^\b\phi_n|^2\les \phi_n^2+ r^{\a}(1+\tau)^{-1-\a}\ep^2E_0,\quad \forall |\b|\leq 1.
$$
Recall that $\pa_u=\pa_t-\pa_r$ and $|\phi_n|^2, |\pa_t\phi_n|^2\les
(1+r)^{-1}(1+\tau)^{-1-\a}\ep^2E_0$. We can estimate
\begin{equation*}
 |\pa_r Z^\b\phi_n|\les
(1+r)^{-\f12}(1+\tau)^{-\f12-\f12\a}\sqrt{E_0}\ep,\quad \forall |\b|\leq 1.
\end{equation*}
Since $\nabb=\frac{\Om}{r}$, we have shown that outside the cylinder $\{r\leq R\}$
\[
  |\pa Z^\b\phi_n|\les
(1+r)^{-\f12}(1+\tau)^{-\f12-\f12\a}\sqrt{E_0}\ep,\quad \forall |\b|\leq 1.
\]
It remains to control $C^2$ estimates of $\phi_n$. Outside the smaller cylinder $\{r\leq \frac{1}{4}R\}$, we use the equation
 \eqref{iteration}. In fact, we can write
$$\pa_{rr}\phi_{n+1}=F(\pa\phi_n)+\pa_{tt}\phi_{n+1}-\frac{2}{r}\pa_r\phi_{n+1}-
\lap \phi_{n+1}-N(\Phi, \phi_{n+1})-L(\pa\phi_{n+1}).
$$
Since we have already shown that
\begin{align*}
 &|\pa\phi_n|,\quad|\pa_{tt}\phi_{n}|\les E_0\ep^2(1+r)^{-\f12}(1+\tau)^{-\f12-\f12\a},\\
& |\lap\phi_n|\leq r^{-2}|\Om^2\phi_n|\les
E_0\ep^2(1+r)^{-\f12}(1+\tau)^{-\f12-\f12\a},
\end{align*}
we conclude that
\[
 |\pa_{rr}\phi_n|\les \sqrt{E_0}\ep
 (1+r)^{-\f12}(1+\tau)^{-\f12-\f12\a}.
\]
Hence when $r\leq \frac{R}{4}$, we can show that
\[
 |\pa^2\phi_n|\leq |\pa_{rr}\phi_n|+|\pa Z\phi_n|\les\sqrt{E_0}\ep
 (1+r)^{-\f12}(1+\tau)^{-\f12-\f12\a}
\]
For the case when $r\leq \frac{R}{4}$, we rely on elliptic theory. First we have the elliptic equation for $\phi_{n+1}$
\[
 \Delta \phi_{n+1}=F(\pa\phi_n)+\pa_{tt}\phi_{n+1}-N(\Phi, \phi_{n+1})-L(\pa\phi_{n+1}).
\]
Since we have shown from \eqref{Caest} that
\[
 \|\pa Z^\b\phi_n\|^2_{C^\f12(B_{\f12 R})}\leq \|\nabla Z^\b\phi_n\|_{C^\f12(B_{\f12 R})}^2+\|Z^{\b+1}\phi_n\|_{C^\f12(B_{\f12 R})}^2
\les E_0\ep^2(1+\tau)^{-1-\a},\quad \forall |\b|\leq 1,
\]
we can show that the right hand side of the above elliptic equation
is uniformly bounded in $C^{\f12}(B_{\f12 R})$. Then Schauder
estimates \cite{elliptic} imply that
\[
 \|\phi_n\|_{C^{2, \f12}(B_{\frac{1}{4}R})}\les E_0\ep^2(1+\tau)^{-1-\a}.
\]
In particular, we have shown that
\[
 \sum\limits_{|\b|\leq 2}|\pa^{\b}\phi_n|\les \sqrt{E_0}\ep
 (1+r)^{-\f12}(1+\tau)^{-\f12-\f12\a}.
\]

Now the classical local existence theory ~\cite{johndelay} shows
that there exists a time $t^*>0$ and a unique smooth solution $\phi(t,
x)\in C^\infty([0, t^*)\times \mathbb{R}^{3})$ of equation
~\eqref{THEWAVEEQ}. Moreover
\[
\phi_{n}(t, x)\rightarrow \phi(t, x)
\]
in $C^{\infty}([0, t^*)\times \mathbb{R}^3)$. Therefore
\[
\sum\limits_{|\b|\leq 2}|\pa^{\b}\phi|\les \sqrt{E_0}\ep
(1+r)^{-\f12}(1+\tau)^{-1+\f12\a},\quad\forall (t, x)\in [0,
t^*)\times \mathbb{R}^3.
\]
By a theorem of H$\ddot{o}$rmander ~\cite{hormander} that as long as
the solution is bounded up to the second order derivatives, the
solution exists globally. That is there exists a unique global
solution $\phi(t,x)\in C^\infty(\mathbb{R}^{3+1})$ which solves
~\eqref{THEWAVEEQ}. In addition since
\[
\phi_{n}(t, x)\rightarrow \phi(t, x),\quad (t, x)\in
\mathbb{R}^{3+1},
\]
$\phi$ obeys all the estimates of $\phi_n$ obtained above. We thus
finished the proof of Theorem \ref{maintheorem2}.

\bigskip

To prove Theorem \ref{maintheorem}, it suffices to check that the
functions $\Phi$, $\mathcal{N}^\mu(\pa\Phi)$,
$\mathcal{N}^{\mu\nu}(\pa\Phi)$ satisfy the conditions in Theorem
\ref{maintheorem2} that $\Phi$, $L^\mu$, $h^{\mu\nu}$ satisfy
correspondingly. In fact, notice that
\[
|\pa Z^\b \mathcal{N}^{\mu}(\pa\Phi)|\leq C(\mathcal{N})|\pa^2
Z^\b\Phi|,\quad \forall |\b|\leq 2.
\]
The boundedness of $\pa^2 Z^\b \Phi$, for all $|\b|\leq 2$, follows
from the equation for $Z^\b\Phi$(similarly, express the only unknown
term $\pa_{rr}Z^\b\Phi$ as a combination of terms with known
$L^\infty$ norm). Hence $|\pa Z^\b\mathcal{N}(\pa\Phi)|+|\pa^2 Z^\b
\Phi|$ is bounded by a constant depending on $C_0$ and the
nonlinearity $\mathcal{N}$. For the other conditions, when $t\leq
t_0$, we have
\[
|Z^\b
\mathcal{N}^\mu(\pa\Phi)|+|Z^\b\mathcal{N}^{\mu\nu}(\pa\Phi)|\leq
C(\mathcal{N}, C_0),\quad \forall |\b|\leq 4.
\]
When $t\geq t_0$, we can show
\begin{align*}
&|Z^\b\mathcal{N}^{\mu}|\leq C(\mathcal{N},
C_0)(1+|x|)^{-1-\f12\a_0}(1+(t-|x|)_{+})^{-1},\\
&|Z^\b\mathcal{N}^{\mu\nu}(\Phi)|\leq C(\mathcal{N},
C_0)(1+|x|)^{-\f12 \a_0},\quad |\b|\leq 4.
\end{align*}
Replace $\a$ with $\min\{\frac{\a_0}{6}, \a\}$. Then the functions
$\mathcal{N}^{\mu}(\pa\Phi)$, $\mathcal{N}^{\mu\nu}(\pa\Phi)$
satisfy the conditions in Theorem \ref{maintheorem2}. We thus have
the stability result of Theorem \ref{maintheorem}.

\bibliography{shiwu}{}

\begin{thebibliography}{10}

\bibitem{alinhac-sls}
S.~Alinhac.
\newblock Stability of large solutions to quasilinear wave equations.
\newblock {\em Indiana Univ. Math. J.}, 58(6):2543--2574, 2009.

\bibitem{ChDNull}
D.~Christodoulou.
\newblock Global solutions of nonlinear hyperbolic equations for small initial
  data.
\newblock {\em Comm. Pure Appl. Math.}, 39(2).

\bibitem{kcg}
D.~Christodoulou and S.~Klainerman.
\newblock {\em The global nonlinear stability of the Minkowski space},
  volume~41 of {\em Princeton Mathematical Series}.
\newblock Princeton University Press, Princeton, NJ, 1993.

\bibitem{dr3}
M.~Dafermos and I.~Rodnianski.
\newblock The redshift effect and radiation decay on black hole spacetimes.
\newblock {\em Comm. Pure Appl. Math.}, 62(7):859--919, 2009.

\bibitem{newapp}
M.~Dafermos and I.~Rodnianski.
\newblock A new physical-space approach to decay for the wave equation with
  applications to black hole spacetimes.
\newblock In {\em X{VI}th {I}nternational {C}ongress on {M}athematical
  {P}hysics}, pages 421--432. World Sci. Publ., Hackensack, NJ, 2010.

\bibitem{elliptic}
D.~Gilbarg and N.Trudinger.
\newblock {\em Elliptic Partial Differential Equations of Second Order}.
\newblock Springer-Verlag, Berlin, reprint of the 1998 edition edition, 2001.

\bibitem{hormander}
L.~H{\"o}rmander.
\newblock {\em Lectures on Nonlinear Hyperbolic Differential Equations}.
\newblock Springer-Verlag, Berlin, 1997.

\bibitem{johndelay}
F.~John.
\newblock Delayed singularity formation in solutions of nonlinear wave
  equations in higher dimensions.
\newblock {\em Comm. Pure Appl. Math.}, 29(6):649--682, 1976.

\bibitem{johnblowup}
F.~John.
\newblock Blow-up for quasilinear wave equations in three space dimensions.
\newblock {\em Comm. Pure Appl. Math.}, 34(1):29--51, 1981.

\bibitem{klinvar}
S.~Klainerman.
\newblock Uniform decay estimates and the lorentz invariance of the classical
  wave equation.
\newblock {\em Comm. Pure Appl. Math.}, 38(3):321--332, 1985.

\bibitem{klnull}
S.~Klainerman.
\newblock The null condition and global existence to nonlinear wave equations.
\newblock In {\em Nonlinear systems of partial differential equations in
  applied mathematics, Part 1 (1984)}, volume~23 of {\em Lectures in Appl.
  Math.}, pages 293--326. Amer. Math. Soc., Providence, RI, 1986.

\bibitem{klmulti}
S.~Klainerman and T.~Sideris.
\newblock On almost global existence for nonrelativistic wave equations in 3d.
\newblock {\em Comm. Pure Appl. Math.}, 49(3).

\bibitem{gx-lindblad2}
H.~Lindblad.
\newblock Global solutions of quasilinear wave equations.
\newblock {\em Amer. J. Math.}, 130(1):115--157, 2008.

\bibitem{SMigor}
H.~Lindblad and I.~Rodnianski.
\newblock The global stability of {M}inkowski space-time in harmonic gauge.
\newblock {\em Ann. of Math. (2)}, 171(3):1401--1477, 2010.

\bibitem{sogge-metcalfe-nakamura}
J.~Metcalfe, M.~Nakamura, and C.~D. Sogge.
\newblock Global existence of solutions to multiple speed systems of
  quasilinear wave equations in exterior domains.
\newblock {\em Forum Math.}, 17(1):133--168, 2005.

\bibitem{sogge-metcalfe2}
J.~Metcalfe and C.~Sogge.
\newblock Long-time existence of quasilinear wave equations exterior to
  star-shaped obstacles via energy methods.
\newblock {\em SIAM J. Math. Anal.}, 38(1):188--209 (electronic), 2006.

\bibitem{sogge-metcalfe}
J.~Metcalfe and C.~Sogge.
\newblock Global existence of null-form wave equations in exterior domains.
\newblock {\em Math. Z.}, 256(3):521--549, 2007.

\bibitem{mora2}
C.~S. Morawetz.
\newblock Time decay for the nonlinear klein-gordon equations.
\newblock {\em Proc. Roy. Soc. Ser. A}, 306:291--296, 1968.

\bibitem{sideris-multispeed}
T.~Sideris and S.-Y. Tu.
\newblock Global existence for systems of nonlinear wave equations in 3{D} with
  multiple speeds.
\newblock {\em SIAM J. Math. Anal.}, 33(2):477--488 (electronic), 2001.

\bibitem{soggemulti}
C.~Sogge.
\newblock Global existence for nonlinear wave equations with multiple speeds.
\newblock In {\em Harmonic analysis at {M}ount {H}olyoke ({S}outh {H}adley,
  {MA}, 2001)}, volume 320 of {\em Contemp. Math.}, pages 353--366. Amer. Math.
  Soc., Providence, RI, 2003.

\bibitem{sogge}
C.~D. Sogge.
\newblock {\em Lectures on Non-linear Wave Equations}.
\newblock International Press, Boston, MA, second edition, 2008.

\bibitem{yang1}
S.~Yang.
\newblock {Global solutions to nonlinear wave equations in time dependent
  inhomogeneous media}.
\newblock 2010.
\newblock ar{X}iv:math.{A}{P}/1010.4341.

\end{thebibliography}
\bibliographystyle{plain}
\bigskip

Department of Mathematics, Princeton University, NJ 08544 USA

\textsl{Email}: shiwuy@math.princeton.edu

\end{document}